\documentclass{article}
\usepackage{marginnote}
\usepackage[english]{babel}
\usepackage{graphicx}
\usepackage{framed}
\usepackage[normalem]{ulem}
\usepackage{amsmath}
\usepackage{amsthm}
\usepackage{amssymb}
\usepackage{amsfonts}
\usepackage{enumerate}
\usepackage[utf8]{inputenc}
\usepackage{hyperref}
\usepackage{bibentry}
\usepackage[printwatermark]{xwatermark}
\usepackage{color}

\newtheorem{theorem}{Theorem}[section]
\newtheorem{corollary}{Corollary}[section]
\newtheorem{definition}{Definition}[section] 
\newtheorem{lemma}{Lemma}[section]

\newtheorem{remark}{Remark}
\newtheorem{assumption}{Assumption}

\title{A priori error analysis for one-parameter Lagrange and POD reduced order methods}

\author{Antti Hannukainen\thanks{Department of Mathematics and Systems Analysis, Aalto University. E-mail:antti.hannukainen@aalto.fi vigdis.toresen@aalto.fi} 
\and Vigdis Toresen\footnotemark[1] }

\date{\today}

\begin{document}

\maketitle
% REQUIRED
\begin{abstract}
This work deals with the numerical solution of parametric linear elliptic partial differential equations using the finite element method for spatial discretization and a proper orthogonal decomposition or Lagrange subspace method based ROM for treating the parameter dependency. We study a simple one-parameter model problem and prove an error estimate for these ROMs as a function of the snapshot point set. Sub-exponential convergence with respect to the number of snapshot points is established under very mild assumptions on the point set. Exponential convergence is proven for structured point sets. In contrast to existing error analysis of Lagrange subspace methods, our estimates are not based on interpolation or perturbation arguments. Rather, we define a new approximation operator from the snapshot subspace and prove error and stability estimates for it. 
\end{abstract}

% REQUIRED

% REQUIRED

\section{Introduction}
This work deals with the numerical solution of parametric partial differential equations, that is, PDEs with additional variables that represent varying material properties, geometry, or boundary conditions. Such equations arise, e.g., in uncertainty quantification \cite{BaNoTe:07}, design optimisation \cite{MaQuRo:12,LaRo:10}, and solution of inverse problems \cite{FaMaYoWiBl:11,LiWiGh:10}. These applications require multiple solutions of the parametric PDE that typically consume a large portion of the computational time.

We consider linear elliptic parametric PDEs and a solution strategy in which the problem is first discretised with respect to spatial variables using the finite element method (FEM). Then, a reduced order model (ROM) for the resulting parametric linear system is constructed using the proper orthogonal decomposition (POD) method \cite{KaVo:07} or the Lagrange subspace method. These ROMs are designed so that they can be evaluated much faster than the full model, thus speeding up computations that require several solutions of a parametric PDE. There are other approaches to achieve similar speed-up, e.g., one can apply the Galerkin method simultaneously in both spatial and parameter dimensions \cite{ScGi:11}, solve the linear system using collocation methods \cite{BaNoTe:07}, or use other snapshot-based approaches such as the greedy method \cite{CoDe:15}. 

It is natural to define a ROM for a linear elliptic PDE as a Galerkin projection to a low-dimensional reduced subspace that can approximate the parametric solution to the desired tolerance. The POD and Lagrange subspace methods create reduced subspaces by first solving the problem for a set of snapshot parameters. In the POD method, the singular value decomposition is then used to find a reduced subspace based on these snapshot solutions. In the Lagrange method, the subspace is defined as the span of the snapshot solutions. Both methods are easy to implement in existing finite element (FE) codes. If the parameter dependency is affine, the FE assembly only needs to be made once.

This article gives a priori error bounds for the Lagrange and POD methods as a function of the snapshot set. Such estimates are interesting as they can help in choosing the snapshots so that the error of the ROM is within a given tolerance while keeping the number of solution evaluations required to construct the model feasible. A priori error estimates for the Lagrange subspace method have been studied in \cite{MaPaTu:02} and for the greedy method in \cite{CoDe:15}. The POD has been analysed in \cite{KaVo:07}. There also exists error analysis for reduced basis methods in a more general non-linear setting, see \cite{Po:85,FiRh:83}. These earlier works are reviewed in Section~\ref{sec:review}. 

We consider a one-parameter model problem and estimate the error of the related POD or Lagrange ROM as a function of the snapshot point set in a weighted $L^2$-norm. The novelty of our analysis is a new method for approximating a function $f:(-1,1) \rightarrow \mathbb{R}^m$ as a weighted sum of snapshot solutions, i.e., $f(y) \approx \sum_{i=1}^N \alpha_i(y) f(y_i)$. Here, $\{y_i\}_{i=1}^N$ are snapshot points and  $\alpha_i:(-1,1)\rightarrow \mathbb{R}$ are the coefficient functions. We bound the approximation error by the product of a stability constant and the best degree-$n$ polynomial approximation error of $f$. The polynomial degree $n$ is determined by a condition on the snapshot point set. Namely, we assume that there exists a quadrature rule using snapshot points and positive weights that can integrate any degree $2n$ polynomial with relative accuracy smaller than one. Unlike the Lebesgue constant, which grows exponentially for equidistant point sets, the stability constant of our approximation grows at sub-linear rate $\sqrt{n}$ under mild assumptions on the point set. These approximation results are general and do not depend on the ROM setting. 

The earlier research on approximating functions based on their point evaluations on arbitrary point sets aims to approximate $f$ from a given subspace $V_m$ and relate the error to the best approximation error of $f$ from the same subspace $V_m$. The stability constant is controlled by increasing the number of point evaluations. In particular, \cite{An2022,An2025} considers polynomial approximation and \cite{Cohene2013,Cohen2017} allow $V_m$ to be any suitable subspace. Our approach is fundamentally different as we construct the approximation from the snapshot subspace, but relate the error to the best polynomial approximation of $f$ similar to Lagrange interpolation. The authors are not aware of any previous work that considers such estimates for arbitrary point sets. 

We apply the approximation operator to derive a priori error estimates for Lagrange and POD ROMs. In particular, we show that the error is bounded by the best degree $n$ polynomial approximation error of the parameter-to-solution map in a weighted $L^\infty$-norm.  We establish sub-exponential convergence of POD and Lagrange ROM with respect to the number of snapshot points under very mild assumptions on the point set. Exponential convergence is established for structured point sets.

In contrast to existing error estimates for POD and Lagrange ROMs, our error analysis is not based on classical Lagrange interpolation or perturbation arguments. Rather, we explicitly construct an approximation to the solution from the snapshot subspace and prove a stability estimate for it. The earlier results on approximating functions based on their pointwise values in \cite{An2022,Cohen2017,Cohene2013,An2025} also establish the $L^2$-approximations, but not as a linear combination of the snapshot solutions. Hence, these previous methods are not directly applicable to POD or Lagrange ROM error analysis.

This work is structured as follows. In Section~2 we discuss the one-parameter model problem. In Section~3 we review previous results on snapshot point selection and their application to the model problem. In Section~4 we define the approximation operator and study its properties. The results in this section are independent of the model problem. Section~5 applies the approximation operator to study the accuracy of the Lagrange and POD ROMs and contains our main results for the model problem. We end the article with discussion and conclusions.

\section{Model problem}
\label{sec:bg}
In this section, we first state our model problem and discuss its FE discretization. Second, following \cite{MaPaTu:02}, we derive a formula for the parameter-to-solution map.

Let $\Omega \subset \mathbb{R}^d$ for $d=2$ or 3 be a domain with Lipschitz boundary. In this work, we consider the one-parameter elliptic PDE: for $y\in (-1,1)$ find $u(y) \in H^1_0(\Omega)$ satisfying 
\begin{equation}
\label{eq:weak_problem}
a(y;u(y),v) = L(v) \quad \mbox{for any $v\in H^1_0(\Omega)$.}
\end{equation}
We make the following assumptions on $a$ and $L$. 
\begin{assumption}
\label{ass:bilin}Let $L:H^1_0(\Omega) \rightarrow \mathbb{R}$ be a bounded linear functional and 
\begin{equation*}
    a(y;v,z) = a_0(v,z) + y a_1(v,z),
\end{equation*}
for any $y\in(-1,1)$ and $v,z \in H^1_0(\Omega)$. Here $a_0,a_1 : H^1_0(\Omega) \times H^1_0(\Omega) \rightarrow \mathbb{R}$ are symmetric, continuous bilinear forms. In addition, assume that $a_0(\cdot,\cdot)$ is coercive (positive), $a_1(\cdot,\cdot)$ is positive semi-definite, and that there exists $0 < r < 1$ such that
\begin{equation}
\label{eq:def_r}
\frac{a_1(v,v)}{a_0(v,v)} \leq 1-r \quad \mbox{for any $v\in H^1_0(\Omega) \setminus \{0\}$.}
\end{equation}
\end{assumption}
Under Assumption~\ref{ass:bilin} there exists $C>0$ such that
\begin{equation} 
\label{eq:a1_cont}
|a_1(v,z)| \leq C a_0(v,v)^{1/2} a_0(z,z)^{1/2},
\end{equation}
for any $v,z \in H^1_0(\Omega)$.
Parametric PDEs that satisfy Assumption~\ref{ass:bilin} arise, e.g., when modelling the electrical potential in a domain with constant background conductivity and an  inclusion $\omega \subset \Omega$. In that case $a_0(v,z) = \int_{\Omega} \sigma_0 \nabla v \cdot \nabla z$ and $a_1(v,z) = \int_{\omega} \sigma_s \nabla v \cdot \nabla z$ for $\sigma_0,\sigma_s \in \mathbb{R}^+$. 

By Assumption~\ref{ass:bilin}, $a(y;\cdot,\cdot)$ is continuous and coercive, i.e., satisfies 
\begin{equation*}
    a(y;v,v) \geq r a_0(v,v) \quad \mbox{for any $v\in H^1_0(\Omega)$}.
\end{equation*}
Hence, the Lax-Milgram theorem guarantees the existence of a unique solution to \eqref{eq:weak_problem} that satisfies the bound 
\begin{equation}
\label{eq:bound_u}
a_0( u(y),u(y))^{1/2} \leq r^{-1} \| L\|,
\end{equation}
for any $y\in (-1,1)$ and $r$ as defined in  \eqref{eq:def_r}. The norm of the linear functional $L$ in~\eqref{eq:bound_u} is defined as  
\begin{equation}\label{eq:lin_func_norm}
\| L \| := \sup_{v \in H^1_0(\Omega) \setminus \{0\}} \frac{ |L(v)| }{a_0(v,v)^{1/2}}.
\end{equation}
Using this definition simplifies the following notation. 
\subsection{Finite element discretization}
The FEM is the standard tool for approximately solving \eqref{eq:weak_problem} for a given $y \in (-1,1)$. In FEM, the weak problem is posed in a finite dimensional subspace $X_h \subset H^1_0(\Omega)$ that is defined using a mesh\footnote{For simplicity, assume that the mesh conforms to the domain $\Omega$.} of the domain $\Omega$, see, e.g., \cite{BrSc:08}. That is, one solves the discretised problem: for $y\in (-1,1)$ find $u_h(y) \in X_h$ such that 
\begin{equation}
\label{eq:weak_FE}
      a(y;u_h(y),v_h) = L(v_h) \quad \mbox{for any $v_h \in X_h$.}
\end{equation}
The FE solution satisfies the same bound as the exact solution $u(y)$, i.e.,
\begin{equation}
\label{eq:bound_uh}
a_0(u_h(y),u_h(y))^{1/2} \leq r^{-1} \| L \|,
\end{equation}
for any $y\in (-1,1)$ and $r$ as defined in \eqref{eq:def_r}.

Let $\{\varphi_i\}_{i=1}^m$ be a basis of $X_h$. Then problem \eqref{eq:weak_FE} is equivalent to solving the linear system: for $y\in (-1,1)$ find $x(y) \in \mathbb{R}^{m}$ satisfying
\begin{equation}
\label{eq:plinsys}
    A(y)x(y) = b.
\end{equation}
Here $A\in\mathbb{R}^{m\times m}$ and $b \in \mathbb{R}^m$ satisfy $[A(y)]_{ij} = a(y;\varphi_i,\varphi_j)$ and $b_i = L(\varphi_i)$. The solutions to \eqref{eq:weak_FE} and \eqref{eq:plinsys} are related as
\begin{equation}
\label{eq:duality}
    u_h(y) = \sum_{i=1}^m x_i(y) \varphi_i. 
\end{equation}
It follows from the definition of $A(y)$ that 
\begin{equation}
\label{eq:quad2bilin_A}
    \eta^T A(y) \eta = a(y;z,z) \quad \mbox{for any $\eta \in \mathbb{R}^m$ and $z=\sum_{i=1}^m \eta_i \varphi_i$.}
\end{equation}
As $a(y;\cdot,\cdot)$ is symmetric and coercive, the matrix $A(y)$ is symmetric and positive definite. Further, because the parameter dependency of $a(y;\cdot,\cdot)$ is affine, it holds that 
\begin{equation*}
A(y) = A_0 + y A_1, 
\end{equation*}
for $A_0,A_1\in\mathbb{R}^{m\times m}$ such that 
\begin{equation*}
[A_0]_{ij} = a_0(\varphi_i,\varphi_j)
\quad \mbox{and} \quad 
[A_1]_{ij} = a_1(\varphi_i,\varphi_j).
\end{equation*}
Similar to \eqref{eq:quad2bilin_A} it holds that
\begin{equation}
\label{eq:quad2bilin}
\eta^T A_0 \eta = a_0(z,z) \quad \mbox{and}  \quad \eta^T A_1 \eta = a_1(z,z)
\end{equation}
for any $\eta \in \mathbb{R}^m$ and $z=\sum_{i=1}^m \eta_i \varphi_i$. Hence, by Assumption~\ref{ass:bilin} $A_0$ is positive definite and $A_1$ is positive semidefinite. 

In the rest of this work, we consider POD and Lagrange subspace ROMs for the linear system \eqref{eq:plinsys}. Before discussing these methods in detail, we give an explicit formula for the parameter-to-solution map $x:(-1,1) \rightarrow \mathbb{R}^m$ and review ROMs based on Galerkin projection to a reduced subspace. 

\subsection{Solution formula}

We proceed by recalling the explicit formula for $x(y)$ from \cite{MaPaTu:02}. This formula is important as it is used in Sections~\ref{sec:review} and \ref{sec:error} to derive a priori error estimates of the POD and Lagrange ROMs of the model problem. First, find an $A_0$- and $A_1$-orthogonal basis for $\mathbb{R}^m$ by solving the eigenproblem: find $(\lambda_k,v_k) \in \mathbb{R}^+ \times \mathbb{R}^m$ satisfying
\begin{equation}
\label{eq:GEV}
A_1 v_k = \lambda_k A_0 v_k.
\end{equation}
This eigenproblem is well posed since $A_0$ is a symmetric positive definite (s.p.d.) matrix. Hence, we choose $\{v_k\}$ to be an $A_0$-orthonormal eigenbasis of $\mathbb{R}^m$. By \eqref{eq:GEV} the eigenvectors $v_k$ are also $A_1$-orthogonal. The connection between the quadratic forms of $A_0,A_1$ and the bilinear form in \eqref{eq:quad2bilin} allows us to estimate the eigenvalues. By \eqref{eq:quad2bilin}, we have that 
\begin{equation*}
    \lambda_k = \frac{v_k^T A_1 v_k}{v_k^T A_0 v_k} = \frac{a_1(z
    _k,z_k)}{a_0(z_k,z_k) } \quad \mbox{ for $z_k = \sum_{i=1}^m (v_k)_i \varphi_i$}.
\end{equation*}
Hence, Assumption~\ref{ass:bilin} gives the bound
\begin{equation}
\label{eq:eigenbound}
0 \leq \lambda_k \leq 1-r \quad \mbox{for any $k\in \{1,\ldots,m\}$.}
\end{equation}
Problem \eqref{eq:plinsys} is straightforward to solve using the eigenbasis. We obtain 
\begin{equation}
\label{eq:solution_formula}
    x(y) = \sum_{k=1}^m \gamma_k(y) \beta_k {v}_k \quad\mbox{for} \quad \gamma_k(y) = \frac{1}{1+y\lambda_k} \quad \mbox{and} \quad \beta_k = {v}_k^T {b}.
\end{equation}
Note that $b = A_0 x(0)$ so that $\beta_k = {v}_k^T A_0 {x}(0)$ and $\sum \beta^2_k = \| {x}(0) \|_{A_0}^{2}$.

\section{Review of subspace methods}
\label{sec:review}
In this section we first recall known results for subspace methods and their application to the model problem. Then the existing approaches for deriving a priori error bounds for Lagrange and POD ROMs are discussed. Finally, we review snapshot based ROMs and their error analysis. 

\subsection{Subspace methods}

It is natural to define ROMs for \eqref{eq:plinsys} by using the Galerkin solution from a reduced subspace. In this section, we briefly recall the subspace solution of linear systems. These solutions satisfy the important best approximation property, which we utilise in ROM a priori error analysis. 

Let $V\subset \mathbb{R}^m$ be a subspace with $\dim V = k$ and $Q \in \mathbb{R}^{m \times k}$ a basis matrix of $V$, i.e., the columns of $Q$ are a basis of $V$. In a subspace method, the solution to (\ref{eq:plinsys}) is approximated as $\tilde{x}(y) = Q z(y) $, where $z(y)\in \mathbb{R}^{k}$ is the solution to the surrogate system
\begin{equation}
\label{eq:surrogate_problem}
Q^T  A(y) Q z(y) = Q^T b.
\end{equation}
Before giving an estimate for $\tilde{x}(y)-x(y)$, we introduce some notation: for s.p.d. $B\in\mathbb{R}^{n\times n}$ denote 
\begin{equation}
    \label{eq:Bnorm}
    (w,z)_B = w^T B z \quad \mbox{for any $w,z \in \mathbb{R}^m$ and} \quad \|\cdot\|_B = \sqrt{(\cdot,\cdot)_B}.
\end{equation}
As $B$ is s.p.d., $(\cdot,\cdot)_B :\mathbb{R}^m \times \mathbb{R}^m \rightarrow \mathbb{R}$ is an inner product and $\| \cdot \|_B$ the induced norm. 

The subspace solution $\tilde{x}(y) \in V$ is the $(\cdot,\cdot)_{A(y)}$-orthogonal projection of $x(y)$ to $V$. Thus, the error $\| x(y)-\tilde{x}(y)\|_{A(y)}$ satisfies the best approximation property
\begin{equation}
\label{eq:best_approximation}
    \| x(y) - \tilde{x}(y) \|_{A(y)} = \min_{\tilde{v}\in V} \| x(y) - \tilde{v} \|_{A(y)} \quad \mbox{for any $y\in(-1,1)$}.
\end{equation}
That is, the accuracy of the subspace solution $\tilde{x}(y)$ depends on $V$. The aim of subspace methods is to seek a low-dimensional $V$ that can approximate $x$ with sufficient accuracy. In this case, the surrogate problem in \eqref{eq:surrogate_problem} has a small dimension and $\tilde{x}(y)$ can be computed much faster than $x(y)$. When solving parametric problems, the same subspace is used for every $y \in (-1,1)$. 

The best approximation property is stated in the $A(y)$-norm. Next, we give norm equivalences that allow us to obtain a near-best approximation result in the $A_0$-norm. The norm equivalences are established in $H^1_0(\Omega)$ and then extended using \eqref{eq:quad2bilin_A} and \eqref{eq:quad2bilin} to vector norms. By Assumption~\ref{ass:bilin} and \eqref{eq:a1_cont} it holds that  
\begin{equation}
|a(y;v,w)| \leq (1+C) a_0(v,v)^{1/2} a_0(w,w)^{1/2}
\end{equation}
for any $v\in H^1_0(\Omega)$ and $y\in (-1,1)$. Then 
\begin{equation}
\label{eq:norm_eq}
r a_0(v,v) \leq a(y;v,v) \leq (1+C) a_0(v,v), 
\end{equation}
for any $v\in H^1_0(\Omega)$ and $y\in (-1,1)$. It follows from the above equation and \eqref{eq:quad2bilin_A}, \eqref{eq:quad2bilin}, \eqref{eq:Bnorm} that
\begin{equation*}
r \| \eta\|^2_{A_0} \leq \| \eta\|^2_{A(y)} \leq (1+C)\| \eta\|^2_{A_0}, 
\end{equation*}
for any $y\in(-1,1)$ and $\eta \in \mathbb{R}^m$. Using this norm equivalence and the best approximation property \eqref{eq:best_approximation} yields
\begin{equation}
\label{eq:A0approx}
\| x(y) - \tilde{x}(y) \|_{A_0} \leq \kappa \min_{\tilde{v} \in V} \| x(y) - \tilde{v} \|_{A_0} \quad \mbox{for any $y\in(-1,1)$},
\end{equation}
where 
\begin{equation}
\label{eq:kappadef}
    \kappa^2 = \frac{1+C}{r},
\end{equation}
for $C,r$ as defined in \eqref{eq:a1_cont} and \eqref{eq:def_r}.

\subsection{Lagrange subspace method}
\label{sec:lagrange}
This section describes the Lagrange subspace method for creating ROMs for \eqref{eq:plinsys} and recalls the error estimates for it from the literature.
Most of the material is well known in ROM literature, see, e.g.,  \cite{MaPaTu:02,Po:85,FiRh:83}, and adapted here to our model problem and notation. Our aim is to collect convergence results obtained using different techniques and evaluate how the convergence rates depend on the problem parameters $r,C$ and what kind of assumptions are made on the snapshot point set. 

Let $\{y_i\}_{i=1}^N \subset (-1,1)$ be the set of distinct sample points. In the Lagrange subspace method, one sets
\begin{equation}
\label{eq:Lagrange}
V(\{y_i\}_{i=1}^N) = \mathop{span} \{ x(y_1), x(y_2),\ldots,x(y_N)\}.
\end{equation}
When the sample point set is clear from the context, we use the shorthand notation $V$. For a given parameter $y\in (-1,1)$ the Lagrange ROM returns the subspace solution $\tilde{x}(y)$ of \eqref{eq:plinsys} from $V$. According to \eqref{eq:A0approx}, it holds that 
\begin{equation*}
\|x(y)-\tilde{x}(y) \|_{A_0} \leq
\kappa\| x(y)-\tilde{v}\|_{A_0}, \hspace{1mm} \text{ for any $\tilde{v} \in V$ and $y\in (-1,1)$}.
\end{equation*}
Error estimates are derived by choosing $\tilde{v} \in V$ so that the right-hand side of the above equation can be bounded.

\subsubsection{Perturbation analysis} 
Next, we consider choosing $\tilde{v} = x(y^*(y))$, where the parameter $y^*(y) \in \{ y_i \}_{i=1}^N$ is the closest snapshot point to the given evaluation point $y$. The error estimate then follows from the Lipschitz continuity of the parameter-to-solution map $x:(-1,1) \rightarrow \mathbb{R}^m$ in $L^\infty(-1,1;\mathbb{R}^m_{A_0})$, where $\mathbb{R}^m_{A_0}$ is the space $\mathbb{R}^m$ with the $A_0$-weighted norm. 
\begin{lemma} 
\label{lemma:lip} Let $x:(-1,1) \rightarrow \mathbb{R}^m$ be the solution to \eqref{eq:plinsys}. Then 
\begin{equation*}
\| x(\eta) - x(\gamma) \|_{A_0} \leq K|\eta-\gamma|,
\end{equation*}
for any $\eta,\gamma \in (-1,1)$ and $K = C r^{-2}  \| L \|$. Here $C,r>0$ are as defined in Assumption~\ref{ass:bilin} and 
the norm of $L : H^1_0(\Omega) \rightarrow \mathbb{R}$ is as defined in (\ref{eq:lin_func_norm}).
\end{lemma}
\begin{proof} This result is well known and follows from perturbation analysis, see, e.g., \cite{CoDe:15}. We give here a version adapted to our notation. By \eqref{eq:Bnorm}, \eqref{eq:duality}, and \eqref{eq:quad2bilin} we have that $\| x(\eta)-x(\gamma)\|_{A_0}^2 = a_0(u_h(\eta)-u_h(\gamma),u_h(\eta)-u_h(\gamma))$. Denote $e_h = u_h(\eta)-u_h(\gamma)$. By \eqref{eq:norm_eq} and \eqref{eq:weak_FE}
\begin{equation*}
a_0(e_h,e_h) \leq r^{-1} a(\eta; e_h, e_h) = r^{-1} \left[ a(\gamma;u_h(\gamma),e_h)-a(\eta;u_h(\gamma),e_h) \right].
\end{equation*}
It follows from  Assumption~\ref{ass:bilin} that
\begin{equation*}
a(\gamma;u_h(\gamma),e_h)-a(\eta;u_h(\gamma),e_h)  = (\gamma-\eta) a_1(u_h(\gamma),e_h).
\end{equation*}
Using continuity of $a_1(\cdot,\cdot)$ in \eqref{eq:a1_cont} gives
\begin{equation*}
a_0(e_h,e_h) \leq Cr^{-1} |\eta-\gamma| a_0(u_h(\gamma),u_h(\gamma))^{1/2} 
a_0( e_h, e_h )^{1/2}.
\end{equation*}
The result follows after applying the bound on
$a_0(u_h(\gamma),u_h(\gamma))$ in \eqref{eq:bound_uh} and dividing by $a_0(e_h,e_h)^{1/2}$.
\end{proof}
It follows from \eqref{eq:A0approx} and Lemma~\ref{lemma:lip} that 
\begin{equation*}
\|x(y)-\tilde{x}(y) \|_{A_0} \leq K \kappa \min_{y^*\in \{y_i\}_{i=1}^N} |y-y^*|.
\end{equation*} 
That is, the upper bound depends on how close to the evaluation point $y$ there is a sample point and on the Lipschitz constant $K$. The implication is that reaching $\epsilon$ accuracy in the $L^\infty(-1,1;\mathbb{R}^m_{A_0})$-norm requires $O(\epsilon^{-1})$ sample points. 

\subsubsection{Lagrange interpolation error analysis}

Let the reduced subspace be as in \eqref{eq:Lagrange} and choose $\tilde{v}$ in \eqref{eq:A0approx} as the Lagrange interpolant of $x\in [C([-1,1])]^{m}$. Next, we use the remainder theorem to  derive an error estimate for the resulting Lagrange ROM. Let $\mathcal{P}^{n}([-1,1])$ be the space of degree $n$ polynomials on $[-1,1]$, and  $\Pi_L(\{y_i\}_{i=1}^N):C([-1,1]) \rightarrow \mathcal{P}^{N-1}([-1,1])$ the Lagrange interpolation operator related to $\{y_i\}_{i=1}^N$, i.e., 
\begin{equation}
\label{eq:LagIP}
    (\Pi_L(\{y_i\}_{i=1}^N) f)(y) = \sum_{i=1}^N \ell_i(y) f(y_i) \quad \mbox{for any $f \in C([-1,1])$.} 
\end{equation}
Here $\ell_i \in \mathcal{P}^{N-1}([-1,1])$ are the Lagrange basis polynomials satisfying $\ell_i(y_j) = \delta_{ij}$ for $i,j \in \{1,\ldots,N\}$. For notational convenience, we write $\Pi_L$ for $\Pi_L(\{y_i\}_{i=1}^N)$ whenever the interpolation point set is clear from the context. We slightly abuse the notation and also apply the interpolation operator $\Pi_L$ to $[C([-1,1])]^m$-functions. Next, key results from Lagrange interpolation error analysis for scalar functions are briefly reviewed, see, e.g., \cite{davis}. We start with the remainder theorem 
\begin{theorem} Let $f \in C^N([-1,1])$. Then there exists $\xi \in [-1,1]$ such that
\begin{equation*} 
f - \Pi_L f = \frac{1}{N!} \prod_{i=1}^{N} (y-y_i) f^{(N)}(\xi).
\end{equation*}
\end{theorem}
\begin{proof} See, e.g., \cite{davis} \end{proof}
\noindent It follows from the remainder theorem that
\begin{equation*}
\| f - \Pi_L f \|_{L^\infty(-1,1)} \leq \frac{1}{N!} \| W(y) \|_{L^\infty(-1,1)} \| f^{(N)} \|_{L^\infty(-1,1)},
\end{equation*}
where $W(y) =\prod_{i=1}^N (y-y_i)$ is a monic polynomial of degree $N$. This upper bound is minimised when  the snapshot points are chosen as Chebyshev interpolation points, i.e., the roots of the degree $N$ Chebyshev polynomial of the first kind $T_N$ so that $W=\hat{T}_N$. The monic Chebyshev polynomial $\hat{T}_N$ has the minimal $L^\infty(-1,1)$ norm over all monic polynomials of degree $N$ and $\| \hat{T}_N \|_{L^\infty(-1,1)} = 2^{-(N-1)}$. This leads to the following corollary
\begin{corollary} \label{cor:cheby} Let $f\in C^{N}([-1,1])$ and $\{y_i\}_{i=1}^N$ be the Chebyshev interpolation points. Then 
\begin{equation*}
\| f - \Pi_L f \|_{L^\infty(-1,1)}\leq \frac{1}{2^{N-1} N!} \| f^{(N)} \|_{L^\infty{(-1,1)}}.
\end{equation*}
\end{corollary}
We apply this corollary to give an error estimate for the subspace solution when the snapshot points are chosen as Chebyshev interpolation points. Before proceeding, we estimate the $L^\infty(-1,1)$ norm of the  derivatives of the coefficient function $\gamma_k = (1+y\lambda_k)^{-1}$. By \eqref{eq:eigenbound}, the eigenvalues satisfy
\begin{equation*}
0\leq\lambda_k \leq 1-r, 
\end{equation*}
for $0 < r < 1 $ defined in Assumption~\ref{ass:bilin}. Direct computation gives
\begin{equation}
\label{eq:gammaDN}
\gamma_k^{(N)}(y) =  (-1)^N N! \frac{\lambda_k^{N}}{(1+y\lambda_k)^{N+1}} \quad \mbox{so that} \quad \|\gamma_k^{(N)}\|_{L^\infty(-1,1)} \leq N! \frac{(1-r)^{N}}{r^{N+1}}.
\end{equation}
We arrive at the following estimate: 
\begin{corollary} 
\label{cor:rembound}
Let $\{y_i\}_{i=1}^N$ be the Chebyshev interpolation points and $V(\{y_i\}_{i=1}^N)$ the Lagrange subspace defined in \eqref{eq:Lagrange}. In addition, let $x(y) \in \mathbb{R}^m$ be the solution to \eqref{eq:plinsys} and $\tilde{x}(y) \in V(\{y_i\}_{i=1}^N)$ the subspace solution to \eqref{eq:plinsys} from $V(\{y_i\})_{i=1}^N$. Then 
\begin{equation}
\label{eq:Remainder_error}
\frac{\|x(y)-\tilde{x}(y) \|_{A_0}}{\| x(0)\|_{A_0}} \leq
\frac{\kappa }{2^{N-1}}\frac{(1-r)^{N}}{r^{N+1}} \quad \mbox{for any $y\in(-1,1)$}. 
\end{equation}
\end{corollary}
\begin{proof}
The near-best approximation property yields 
\begin{equation*}
\|x(y)-\tilde{x}(y) \|_{A_0} \leq
\kappa \| (I-\Pi_L)x(y)\|_{A_0}.
\end{equation*}
By the solution formula \eqref{eq:solution_formula} and $A_0$-orthogonality of $\{v_k\}$ it holds that 
\begin{align*}
\| (I-\Pi_L)x(y)\|^2_{A_0}
&= \sum_{k=1}^m [(I-\Pi_L)\gamma_k(y)]^2 \beta_k^2
\\&
\leq \max_{k} \| (I-\Pi_L) \gamma_k \|^2_{L^\infty(-1,1)} \sum_{k=1}^m  \beta_k^2 .
\end{align*}
By Corollary~\ref{cor:cheby},
\begin{equation*}
\| (I-\Pi_L)\gamma_k \|_{L^\infty(-1,1)} \leq \frac{1}{N!} \frac{1}{2^{N-1}} \| \gamma_k^{(N)}\|_{L^\infty(-1,1)}.
\end{equation*}
The proof is completed by application of the estimate \eqref{eq:gammaDN} and recalling that $\sum_{k=1}^m \beta^2_k = \| x(0) \|_{A_0}^{2}$.
\end{proof}
The upper bound \eqref{eq:Remainder_error} is derived under the assumption that the sample points are Chebyshev interpolation points. The error converges to zero with increasing $N$ iff
\begin{equation*}
    \frac{1-r}{2r} < 1 \quad \mbox{this is} \quad r > \frac{1}{3}.
\end{equation*}
This condition considerably limits the applicability of the bound. 
\subsubsection{Near-best Lagrange interpolation error analysis}

Next, we apply a different technique to derive a Lagrange interpolation-based error bound for $\|x(y)-\tilde{x}(y)\|_{A_0}$. Namely, we use the polynomial exactness of the interpolation operator $\Pi_L$ and bound the error by the Lebesgue constant and the best polynomial approximation error of $\gamma_k$. We then apply a problem specific best approximation result to give final convergence rate. 

\begin{definition} Let $\{y_i\}_{i=1}^N \subset (-1,1)$ and $\Pi_L \equiv \Pi_L(\{y_i\}_{i=1}^N)$ be as defined in \eqref{eq:LagIP}. Then the Lebesgue constant related to the point set $\{y_i\}_{i=1}^N$ is 
\begin{equation*}
\Lambda_N(\{y_i\}_{i=1}^N) = \sup_{f\in C([-1,1])} \frac{\|\Pi_L f\|_{L^\infty(-1,1)}}{\|f\|_{L^\infty(-1,1)}}.
\end{equation*}
\end{definition}
When the point set is clear from the context, we use the notation $\Lambda_N \equiv\Lambda_N(\{y_i\}_{i=1}^N)$. The Lebesgue constant has been estimated for different interpolation point families when $N \rightarrow \infty$, see \cite{brutman1997lebesgue} and \cite{trefethen2019}. For example, for the family of equidistant points $\{-1+2\frac{n-1}{N-1}\}_{n=1}^N$, $\Lambda_N$ grows exponentially with $N$ whereas for Chebyshev points  it only has logarithmic growth. 

By construction, $\Pi_L$ is exact for degree $N-1$ polynomials, i.e., $\Pi_L p = p$ for any $p\in \mathcal{P}^{N-1}([-1,1])$. This leads to the well-known near-best approximation result, see, e.g., \cite{trefethen2019}.
\begin{theorem} 
\label{thm:nearbest} Let $\{y_i\}_{i=1}^N \subset (-1,1)$ be distinct point set and $\Pi_L \equiv \Pi_L(\{y_i\}_{i=1}^N)$ be as defined in \eqref{eq:LagIP}. Then there holds that 
\begin{equation}
\label{eq:lag_best}
\| (I-\Pi_L)f\|_{L^\infty(-1,1)} \leq (1+\Lambda_N) \min_{p\in \mathcal{P}^{N-1}([-1,1])}\| f-p \|_{L^\infty(-1,1)}
\end{equation}
for any $f\in C([-1,1])$. 
\end{theorem}
This result gives an alternative way to derive Lagrange interpolation-based ROM error estimates. The bound \eqref{eq:lag_best} does not depend on the higher-order derivatives of the function $f$, but rather on the best polynomial approximation error of $f$. In general, the best polynomial approximation error can be estimated using Jackson inequalities that bound the error by higher-order derivatives, see \cite{davis}.  Alternatively, one can derive improved bounds using problem-specific techniques. 

We proceed to estimate the interpolation error of $\gamma_k$ by using the best approximation results derived for the family of functions $f(t) = (t-c)^{-1}$ for  $ t\in [-1,1]$ and $c > 1$ given in \cite{Rivlin}, see also \cite{BestUniformApproxJM,BestUniformApproxZY}. By these bounds, it holds that 
\begin{equation*}
\min_{p\in \mathcal{P}^{N-1}([-1,1])}\| \gamma_k - p\|_{L^\infty(-1,1)} \leq  \frac{4}{\lambda_k(\tau_k^2-1)^2} |\tau_k|^{N+1} \quad \mbox{where} \quad |\tau_k| \leq \lambda_k,
\end{equation*}
see Appendix \ref{sec:best_poly}. We proceed to simplify the RHS of this estimate. As $|\tau_k| \leq \lambda_k$, we obtain  
\begin{equation*}
\frac{4}{\lambda_k(\tau_k^2-1)^2} |\tau_k|^{N+1} \leq  \frac{4}{(\tau_k^2-1)^2} |\tau_k|^{N}.
\end{equation*}
Using the bound $|\tau_k| \leq \lambda_k \leq 1-r$, recalling that $r<1$, and estimating $(\tau^2_k-1)^2 > r^2(2-r)^2>r^2$ gives
\begin{equation}
\label{eq:gamma_best}
\min_{p\in \mathcal{P}^{N-1}([-1,1])}\| \gamma_k - p\|_{L^\infty(-1,1)} \leq \frac{4}{r^2} (1-r)^{N}.
\end{equation}
We have arrived at an error estimate based on the near-best approximation result. 
\begin{corollary}
Let $\{y_i\}_{i=1}^N \subset (-1,1)$ be the set of distinct snapshot points, $\Lambda_N \equiv \Lambda_N(\{y_i\}_{i=1}^N)$ be the Lebesgue constant related to $\{y_i\}_{i=1}^N$, and $V(\{y_i\}_{i=1}^N)$ the Lagrange subspace defined in \eqref{eq:Lagrange}. In addition, let $x(y) \in \mathbb{R}^m$ be the solution to \eqref{eq:plinsys} and $\tilde{x}(y) \in V(\{y_i\}_{i=1}^N)$ the subspace solution to \eqref{eq:plinsys} from $V(\{y_i\}_{i=1}^N)$. Then 
\begin{equation}
\label{eq:nb_lag}
    \frac{\| x(y) - \tilde{x}(y)\|_{A_0}}{\| x(0)\|_{A_0}} \leq \frac{4\kappa(1+\Lambda_N)}{r^2}(1-r)^{N} \quad \mbox{for any $y\in(-1,1)$}.
\end{equation}
\end{corollary}
\begin{proof}  The proof is identical to the proof of Corollary~\ref{cor:rembound} except for applying the near-best approximation result of Theorem~\ref{thm:nearbest} and the bound \eqref{eq:gamma_best} to estimate $\| (I-\Pi_L) \gamma_k\|_{L^\infty(-1,1)}$. 
\end{proof}
The behaviour of the estimate \eqref{eq:nb_lag} when $N\rightarrow \infty$ depends on the Lebesgue constant. If the point set is such that the Lebesgue constant grows sub-exponentially, e.g., Chebyshev interpolation points, the estimate converges exponentially to zero for any $0< r <1$. However, it does not explain the performance of the Lagrange subspace ROM, e.g., for equidistant sample points, for which the Lebesgue constant blows up. 

\subsubsection{Local interpolation error analysis} 
\label{sec:local}

In this section, we review Lagrange ROM a priori error analysis from \cite{MaPaTu:02} that is based on local interpolation and the remainder theorem. The aim in using local interpolation is to prove convergence result for arbitrary point sets.  Let $y\in (-1,1)$ and $S(y)\subset \{y_i\}_{i=1}^N$ be the subset of points at distance at most $\tfrac{1}{2}\Delta>0$ from $y$, i.e., 
\begin{equation}
\label{eq:setS}
    S(y) = \left\{ v \in \{ y_i \}_{i=1}^N \; | \; |v-y|\leq \frac{1}{2}\Delta \;\right\}.
\end{equation}
The local interpolation operator corresponding to point $y$ is denoted by $\Pi_{\Delta}(y,\{y_i\}_{i=1}^N)$ and defined as
\begin{equation*}
    \Pi_{\Delta}(y,\{y_i\}_{i=1}^N) = \Pi_L(S(y)).  
\end{equation*}
Let $a$ and $b$ be the smallest and largest elements in $S(y)$. By the remainder theorem on $(a,b)$, $\Pi_{\Delta}(y)$ satisfies the error estimate  
\begin{equation*}
    \max_{t\in [a,b]}\|f(t)- (\Pi_{\Delta}(y)f)(t) \|_{L^\infty(a,b)} \leq \frac{1}{N_\Delta!}\max_{t\in{[a,b]}} \left| \prod_{v \in S(y)} (t-v) \right| \| f^{(N_\Delta)} \|_{L^\infty(a,b)},
\end{equation*}
where $N_\Delta$ is the number of elements in $S(y)$. The idea of local interpolation is that by \eqref{eq:setS} it holds that $\max_{t\in{(a,b)}} \left| \prod_{v \in S(y)} (y-v) \right| \leq \Delta^{N_\Delta}$. Thus, by  \eqref{eq:gammaDN}
\begin{equation}
\label{eq:locbound}
    \|(I- \Pi_{\Delta}) \gamma_k \|_{L^\infty(a,b)} \leq   \frac{ \Delta^{N_\Delta} (1-r)^{N_\Delta}}{r^{N_\Delta+1}}.
\end{equation}

\begin{corollary} \label{cor:locbound}
Let $\{y_i\}_{i=1}^N \subset (-1,1)$ be distinct snapshot points and $V(\{y_i\}_{i=1}^N)$ the Lagrange subspace defined in \eqref{eq:Lagrange}. In addition, let $x$ be the solution to \eqref{eq:plinsys}, and $\tilde{x}(y)$ the subspace solution to \eqref{eq:plinsys} from $V(\{y_i\}_{i=1}^N)$. Then 
\begin{equation}
\label{eq:locbound2}
    \frac{\| x(y) - \tilde{x}(y)\|_{A_0}}{\| x(0)\|_{A_0}} \leq 
    \kappa \frac{\Delta^{N_\Delta}(1-r)^{N_\Delta}}{r^{N_\Delta+1}} \quad \mbox{for any $y\in(-1,1)$}.
\end{equation}
Here $N_\Delta$ is the number of elements in $S(y)$.
\end{corollary}
\begin{proof}  The proof is identical to the proof of Corollary~\ref{cor:rembound} except for applying the local interpolation operator and the error bound \eqref{eq:locbound}. %Convergence requires that $\Delta$ satisfies the condition
\end{proof}

The bound given in Corollary~\ref{cor:locbound} eliminates the need to use structured interpolation point sets. However, convergence requires that 
\begin{equation*}
    \Delta (1-r) < r \quad \quad \mbox{so that} \quad \Delta < \frac{r}{1-r}.
\end{equation*}
That is, the required number of points grows rapidly for small $r$. The article \cite{MaPaTu:02} considers positive perturbations and does not address this issue. 

In addition to local interpolation \cite{MaPaTu:02} proposes exponentially spaced snapshot points. In numerical experiments, exponential snapshot points seems to improve the accuracy of the ROM especially if the parameter range is large; see the results in \cite{Ve:03}.
\subsection{Proper orthogonal decomposition}

ROMs based on the Lagrange subspace $V(\{y_i\}_{i=1}^N)$ are not commonly used in practical computations. Instead, one applies POD or the greedy method to find a reduced subspace that yields sufficiently accurate model that has a smaller dimension. The POD method for our model problem is studied in \cite{KaVo:07}. In that work, POD is analysed as a method for finding an optimal $r$-dimensional subspace that minimises the discrete $L^2$-norm of the error
\begin{equation}
\label{eq:DL2error}
\sum_{i=1}^N \|x(y_i)-\tilde{x}_{\rm POD}(y_i)\|_{A_0}^2 w_i,
\end{equation}
over all $r$-dimensional subspaces. Here $\tilde{x}_{\rm POD}$ is the POD subspace solution and $\{w_i\}_{i=1}^N \subset \mathbb{R}^+$ are some positive quadrature weights. Let $X\in \mathbb{R}^{m\times N}$ be the snapshot matrix and $W\in \mathbb{R}^{N\times N}$ be the weight matrix defined as
\begin{equation}\label{eq:XW}
X = \begin{bmatrix} x(y_1) & x(y_2) & \cdots & x(y_N) \end{bmatrix} \quad \mbox{and} \quad W = \mathop{diag}(\sqrt{w_1},\ldots,\sqrt{w_N}).
\end{equation}
In \cite{KaVo:07} it is proven that the $r$-dimensional subspace minimising \eqref{eq:DL2error} is obtained from $r$ left singular vectors of $XW$ with respect to $A_0$-inner product. In addition, \cite{KaVo:07} gives an error estimate for \eqref{eq:DL2error} in terms of singular values of the matrix $XW$ and bounds the difference of POD subspace solutions obtained using different point sets. 

Before proceeding, we need to introduce some notation. Let $L$ be a Cholesky factor of $A_0$ s.t. $A_0 = L^TL$ and define the matrix norm $\| \cdot \|_{2,A_0}$ as
\begin{equation}
\label{eq:2A0norm}
\| B \|_{2,A_0} = \max_{x\in\mathbb{R}^N \setminus \{0\} } \frac{\| Bx \|_{A_0} }{\|x\|_2} = \max_{x\in\mathbb{R}^N \setminus \{0\} } \frac{\| LBx \|_{2} }{\|x\|_2} \quad \mbox{for any $B \in \mathbb{R}^{m \times N}$.}
\end{equation}
By definition, $\| B \eta \|_{A_0} \leq \| B \|_{2,A_0} \| \eta \|_2$ for any $\eta \in \mathbb{R}^N$ and $B\in\mathbb{R}^{m\times N}$. The POD subspace is defined as follows. 
\begin{definition} \label{def:VPOD} Let $\{y_i\}_{i=1}^N \subset (-1,1)$, $\{w_i\}_{i=1}^N \subset \mathbb{R}^+$, $X$ and $W$ be as defined in (\ref{eq:XW}) and $\tilde{X} = XW$. In addition, let $\tilde{X}_r$ be the best rank $r$ approximation to $\tilde{X}$ in the $\| \cdot \|_{2,A_0}$ norm, and $\sigma_{r+1}$ such that $\| \tilde{X}-\tilde{X}_r\|_{2,A_0} = \sigma_{r+1}$. Then $V_{\rm POD} \equiv V_{\rm POD}(r,\{y_i\}_{i=1}^N,\{w_i\}_{i=1}^N)$ is defined as 
\begin{equation}
\label{eq:PODdef}
V_{\rm POD}=R(\tilde{X}_r),
\end{equation}
where $R$ denotes the range of a matrix.
\end{definition}
Next, we discuss how the best rank $r$ approximation to $\tilde{X}$ in $\|\cdot\|_{2,A_0}$ -norm is computed. By \eqref{eq:2A0norm} it holds that $\| \tilde{X} - \tilde{X}_r \|_{2,A_0} = \| L \tilde{X} - L \tilde{X}_r \|_2$. Thus, the matrix $\tilde{X}_r$ can be obtained by first finding the best rank $r$ approximation $Y_r$ of $L\tilde{X}$ in the $2$-norm and then setting $\tilde{X}_r = L^{-1} Y_r$. The low rank approximation $Y_r$ is obtained from a truncated singular value decomposition (SVD) of $L\tilde{X}$ as usual. Then 
\begin{equation*}
\label{eq:def_sr}
\|\tilde{X}-\tilde{X}_r\|_{2,A_0} = \sigma_{r+1},
\end{equation*}
where $\sigma_{r+1}$ is the $r+1$ largest singular value of $L\tilde{X}$.

We proceed to consider the special case $W=I$ that leads to a POD variant that can be easily analysed using Lagrange interpolation  technique. Observe, that in this case $\tilde{X} = X$. Other choices of $W$ are considered later in Theorem~\ref{thm:poderror} that gives a POD error estimate in the $L^2(-1,1;\mathbb{R}^m_{A_0})$ norm using our approximation operator. The following Lemma gives an error estimate that links the POD error for $W=I$ to the accuracy of Lagrange interpolation, the Lebesgue constant, and singular value $\sigma_{r+1}$ of $LX$. 
\begin{lemma} Use the same notation as in 
Definition \ref{def:VPOD}. Let $W = I$, $x(y)\in\mathbb{R}^m$ be the solution to \eqref{eq:plinsys}, and $\tilde{x}_{\rm POD}(y)\in V_{\rm POD}$ be the subspace solution to \eqref{eq:plinsys} from $V_{\rm POD}$. Then
\begin{equation*}
\| x(y) - \tilde{x}_{\rm POD}(y)\|_{A_0} \leq \ \kappa \| (I-\Pi_L)x\|_{L^\infty(-1,1;\mathbb{R}^m_{A_0})} + \kappa\sigma_{r+1} \Lambda_N.
\end{equation*}
for any $y\in(-1,1)$. Here $\Lambda_N$ is the Lebesgue constant  related to the point set $\{y_i\}_{i=1}^N$, $\sigma_{r+1}$ is as defined in \eqref{eq:def_sr}, and $\kappa$ as defined in \eqref{eq:kappadef}.
\end{lemma}
\begin{proof} Let $y\in(-1,1)$. By the near-best approximation property in \eqref{eq:A0approx}, the POD solution satisfies 
\begin{equation*}
\| x(y) - \tilde{x}_{\rm POD}(y) \|_{A_0} \leq \kappa \| x(y)-\tilde{v} \|_{A_0} \quad \mbox{for any $\tilde{v} \in V_{\rm POD}$}.
\end{equation*}
By definition of $V_{\rm POD}$, any $\tilde{v} \in V_{\rm POD}$ has the representation $\tilde{v} = \tilde{X}_r \eta$ for some $\eta \in \mathbb{R}^N$. Adding and subtracting $X\eta$ and applying the triangle inequality gives
\begin{equation}
\label{eq:POD_intermediate1}
\| x(y) - \tilde{x}_{\rm POD}(y) \|_{A_0} \leq 
\kappa\| x(y) - X\eta \|_{A_0} + 
\kappa\| (X-\tilde{X}_r) \|_{2,A_0} \|\eta\|_2. 
\end{equation}
for any $\eta \in \mathbb{R}^N$. We choose 
\begin{equation*}
\eta = \begin{bmatrix} \ell_1(y) & \cdots & \ell_N(y)
\end{bmatrix}^T \quad \mbox{so that} \quad  X \eta = (\Pi_L x)(y).
\end{equation*}
Thus, the first term on the RHS of \eqref{eq:POD_intermediate1} satisfies $\| x(y) - X\eta \|_{A_0} = \| [(I-\Pi_L)x](y) \|_{A_0}$. We proceed to estimate the second term. By norm equivalence between vector norms and properties of the Lebesgue constant we have that 
\begin{equation*}
\| \eta \|_2 \leq \| \eta \|_1 = \sum_{i=1}^N | \ell_i(y)| \leq \Lambda_N,
\end{equation*}
see \cite{brutman1997lebesgue}. Recalling that in the case $W=I$ we have $\|X-\tilde{X}_r\|_{2,A_0}  = \sigma_{r+1}$ completes the proof. 
\end{proof}
\subsection{Greedy method}
The greedy method aims to find a subspace that minimises the approximation error,
\begin{equation*}
\sup_{y\in(-1,1)} \| x(y)-\tilde{x}(y)\|_{A_0}, 
\end{equation*}
where $\tilde{x}(y)$ is the greedy subspace solution. A family of greedy subspaces $\{V_k\}$ is constructed from a sample set $\{y_i\}_{i=1}^N \subset (-1,1)$ as follows. On step $k$ one seeks $y^* \in \{y_i\}_{i=1}^N$ that maximizes the error
\begin{equation}
\label{eq:proj_error}
\|x(y_i)-\tilde{x}_{k-1}(y_i)\|_{A_0},
\end{equation}
where $\tilde{x}_{k-1}(y)$ is the subspace solution from the previous greedy subspace $V_{k-1}$. Then $V_k$ is obtained as 
\begin{equation*}
V_{k} = V_{k-1} \oplus \mathop{span}(x(y^*)).
\end{equation*}
The convergence of this process is analysed, e.g., in \cite{CoDe:15}. The analysis utilises Lipschitz continuity of the parameter-to-solution map and establishes that reaching $\epsilon$-accuracy requires the use of an $\epsilon$-net for the sample set $\{y_i\}_{i=1}^N$. This requirement leads to a very large number of sampling points, unless the Kolmogorov $n$-width of the PDE coefficient set is small. 

Evaluating the exact error in \eqref{eq:proj_error} requires multiple solutions of the PDE. The weak greedy variant reduces this computational cost by replacing the exact error by an a posteriori error estimator, see \cite{BiCoDa:11,BuMaPa:12}.

\section{Approximation from the snapshot subspace} \label{sec:approximation}
In this section, we develop a novel method for approximating a function based on its pointwise values. These results are general and not directly related to the ROM setting. The approximation method is used in Section~\ref{sec:error} to derive new Lagrange and POD ROM a priori error estimates under very mild conditions on the sample set. In particular, we design a family of approximation operators $\Pi_n : C([-1,1];\mathbb{R}^m_{A_0}) \rightarrow \mathcal{P}^n([-1,1];\mathbb{R}^m)$ that have the following properties:

\renewcommand{\labelenumi}{(\roman{enumi})}

\begin{enumerate}
\item Snapshot approximation:
\begin{equation}
\label{eq:sn_approx}
    (\Pi_n f)(y) = \sum_{i=1}^N \alpha_i(y) f(y_i),
\end{equation}
for coefficient functions $\alpha_i \in \mathcal{P}^n([-1,1])$ that are independent of $f$ and depend only on the degree $n$ and the set of snapshot points $\{ y_i \}_{i=1}^N$. 
\item Degree-$n$ polynomial exactness:  \begin{equation}
\label{eq:polynomial_exact}
    \Pi_n f  = f
\end{equation}
for any  $f \in \mathcal{P}^n([-1,1];\mathbb{R}^m)$.
\item Stability: 
\begin{equation}
\label{eq:stability}
    \| \Pi_n f\|_{L^2(-1,1;\mathbb{R}^m_{A_0})}\leq C_{\rm s}(n) \| f \|_{L^\infty(-1,1;\mathbb{R}^m_{A_0})}
\end{equation}
for any $f \in C([-1,1];\mathbb{R}^m_{A_0})$. In order to use this bound, we explicitly specify the dependence of $C_{\rm s}(n)$ on $n$.
\end{enumerate}

This section is structured as follows. First, Lemma~\ref{lemma:cond} establishes that $\Pi_n$ is degree-$n$ polynomially exact if the coefficient vector $\alpha(y)$ is a solution to an under-determined linear system. Then, Lemma~\ref{lemma:existence} proves that this linear system has a solution, if the point set $\{y_i\}_{i=1}^N$ satisfies the  Marcinkiewicz–Zygmund (MZ) type inequality
\begin{equation}
\label{eq:cond0}
\max_{q\in \mathcal{P}^n([-1,1])\setminus\{0\}}\left| \frac{ \int_{-1}^1 q(y)^2\;dy-\sum_{i=1}^N q(y_i)^2w_i}{\int_{-1}^1 q(y)^2\;dy } \right| < 1,
\end{equation}
for some positive weights $w_i$. That is, there exists some quadrature rule with positive weights associated to the point set that can integrate second powers of degree $n$ polynomials with sufficient relative accuracy. In addition, Lemma~\ref{lemma:existence} gives an estimate for the norm of $\alpha$ that is needed to derive an upper bound for the stability constant $C_{\rm s}(n)$. Theorem~\ref{thm:mainresult} collects these results and gives a condition for the point set that guarantees the existence of an approximation operator that satisfies (i)-(iii). Near-best approximation result follows immediately and is stated in Corollary~\ref{cor:ip_error}. Then Lemmas~\ref{lemma:stab_gauss} and \ref{lemma:stab_random} study the condition~\eqref{eq:cond0} for the family of Gaussian quadrature points and arbitrarily chosen point sets, respectively.

The challenge in defining the operator $\Pi_n$ is to construct the coefficient vector $\alpha(y)$. Interestingly, in our construction, the elements of the coefficient vector are degree $n$ polynomials that are \emph{independent of $f$} and depend only on the polynomial degree $n$ as well as on the snapshot points $\{y_i\}_{i=1}^N$. The stability property \eqref{eq:stability} follows by bounding the norm of the coefficient vector $\alpha(y)$. It turns out that it is natural to use the norm $\int_{-1}^1 \| W^{-1} \alpha \|_2^2\; dy$ where $W\in \mathbb{R}^{N\times N}$ is a diagonal weight matrix defined as in (\ref{eq:XW}) for $w_i>0$. This matrix is an analytical tool that is not needed in any computation, but the weights appear in the MZ-condition \eqref{eq:MZcond}.

We proceed to study the existence of a coefficient function $\alpha:(-1,1) \rightarrow \mathbb{R}^{N}$ such that the associated approximation operator $\Pi_n$ satisfies \eqref{eq:polynomial_exact} and estimate the norm $\int_{-1}^1\| W^{-1} \alpha \|^2_2\; dy$. The analysis uses the discrete $L^2(-1,1)$ inner product
\begin{equation*}
    (g,f)_D:= \sum_{i=1}^Ng(y_i)f(y_i) w_i.
\end{equation*}
We make the following standing assumption
\begin{assumption} \label{ass:ip} Let $\{y_i\}_{i=1}^N\subset (-1,1)$ be a snapshot point set and $\{w_i\}_{i=1}^N \subset \mathbb{R}_+$ positive weights. Assume that 
\begin{equation*}
(g,f)_D:= \sum_{i=1}^N g(y_i)f(y_i) w_i
\end{equation*}
is an inner product in $\mathcal{P}^n([-1,1])$ and the quadrature rule $(\{y_i\},\{w_i\})$ is exact for constant functions, i.e., $\sum_{i=1}^N w_i = 2$. 
\end{assumption}
\noindent This assumption is very mild, i.e., it states that there must exist $n+1$ unique snapshot points in $\{y_i\}_{i=1}^N$ so that condition $(g,g)_D > 0$ holds in $\mathcal{P}^n([-1,1])$ and $(\cdot,\cdot)_D$ is an inner product in $\mathcal{P}^n([-1,1])$. The assumption on quadrature rule being exact for constant functions is made to obtain simple multiplicative constants in the following analysis. 

Let $\{P_j\}_{j=0}^{n}$ be an $(\cdot,\cdot)_D$-orthogonal basis  of $\mathcal{P}^n([-1,1])$ that is normalized as
\begin{equation}
\label{eq:Pnorm}
    \int_{-1}^1 P_j(y)^2 \;dy =1 \quad \mbox{for any $j \in \{0,\ldots,n\}$}.
\end{equation} 
These polynomials can be constructed using the Gram-Schmidt process similar to Legendre polynomials, see \cite{gautschi2004orthogonal}. We begin with a Lemma that gives a condition for the existence of coefficient function $\alpha \in \mathcal{P}^n([-1,1];\mathbb{R}^N)$ such that the associated approximation operator is degree-$n$ polynomially exact. Let 
\begin{equation*}
    \Phi \equiv \Phi(\{y_i\}_{i=1}^N) = \begin{bmatrix} P_0(y_1) & \cdots & P_0(y_N) \\ \vdots & \vdots& \vdots \\P_n(y_1) & \cdots & P_n(y_N)
    \end{bmatrix} \quad \mbox{and} \quad   
    \psi(y) = \begin{bmatrix} P_0(y) \\ \vdots \\ P_n(y) \end{bmatrix}.
\end{equation*}
The next lemma proves that $\Pi_n$ is degree-$n$ polynomially exact if the coefficient vector $\alpha(y)$ is a solution to an under-determined linear system \eqref{eq:alphacond}.
\begin{lemma}\label{lemma:cond} Assume that for any $y\in(-1,1)$ there exists $\alpha(y) \in \mathbb{R}^N$ such that
\begin{equation}
\label{eq:alphacond}
    \Phi\alpha(y) = \psi(y). \quad 
\end{equation}
Then $\Pi_n$ s.t. $(\Pi_n f)(y) = \sum_{i=1}^N \alpha_i(y) f(y_i)$ is degree-$n$ polynomially exact, i.e., 
\begin{equation*}
    \Pi_n f = f \quad \mbox{for any $f\in \mathcal{P}^n([-1,1];\mathbb{R}^m)$}.
\end{equation*}
\end{lemma}
\begin{proof} Let $f\in \mathcal{P}^n([-1,1];\mathbb{R}^m)$. The condition \eqref{eq:polynomial_exact} has to hold componentwise, i.e. 
\begin{equation}
\label{eq:cond1}
\sum_{i=1}^N \alpha_i(y)  f_k(y_i) = f_k(y) \quad \mbox{for any $y\in(-1,1)$ and $k\in\{1,\ldots,m\}$}.
\end{equation}
Expand the polynomials $f_k$ in the basis $\{P_j\}$ as 
\begin{equation}
\label{eq:Leg}
f_{k}(y) =\sum_{j=0}^n \eta_{kj} P_j(y),
\end{equation}
where $\eta_{kj} \in \mathbb{R}$ are the expansion coefficients.  By expansion \eqref{eq:Leg}, the condition in \eqref{eq:cond1} is equivalent to
\begin{equation*}
 \sum_{j=0}^n \eta_{kj} \sum_{i=1}^N  \alpha_i(y)  P_j(y_i) = 
\sum_{j=0}^n  \eta_{kj} P_j(y)
\end{equation*}
for any $y\in(-1,1)$ and $k\in\{1,\ldots,m\}$. A possible solution  is to choose $\alpha(y)$ satisfying
\begin{equation*}
\sum_{i=1}^N  \alpha_i(y)  P_j(y_i) = 
P_j(y) ,
\end{equation*}
for any $j\in \{0,\ldots,n\}$ and $y\in(-1,1)$. This is the linear system in \eqref{eq:alphacond}.
\end{proof}
Next, we show that there exists a solution to the linear system \eqref{eq:alphacond} if the quadrature rule $(\{y_i\}_{i=1}^N,\{w_i\}_{i=1}^N)$ related to the snapshot point set is sufficiently accurate for second powers of degree $n$ polynomials. In addition, we estimate the norm of $\alpha(y)$. 
\begin{lemma}
\label{lemma:existence} Let $\Phi \in \mathbb{R}^{(n+1)\times N}$, $\psi(y) \in \mathbb{R}^{n+1}$ be as defined in \eqref{eq:alphacond} and define $e \equiv e(\{y_i\}_{i=1}^N,\{w_i\}_{i=1}^N,n)$ as 
\begin{equation}
\label{eq:MZcond}
e = \max_{q\in \mathcal{P}^n([-1,1])\setminus\{0\}}\left| \frac{ \int_{-1}^1 q(y)^2\;dy-\sum_{i=1}^N q(y_i)^2w_i}{\int_{-1}^1 q(y)^2\;dy } \right|.
\end{equation}
If $e<1$, then there exists $\alpha \in \mathcal{P}^n([-1,1];\mathbb{R}^N)$ such that 
\begin{equation}
\label{eq:alpha_ls}
\Phi \alpha(y) = \psi(y) \quad \mbox{for any $y\in (-1,1)$}
\end{equation}
and 
\begin{equation}
\label{eq:alpha_stab}
\int_{-1}^1 \| W^{-1} \alpha(y) \|^2_2 \; dy 
\leq  \frac{n+1}{1-e}.
\end{equation} 

\end{lemma} \begin{proof}
First, observe that as $\psi \in \mathcal{P}^n([-1,1];\mathbb{R}^n)$ and $\Phi$ is independent on $y$, any solution $\alpha$ to \eqref{eq:alphacond} is a degree $n$ polynomial. We proceed to study the existence of a solution. By $(\cdot,\cdot)_D$ orthogonality of $\{P_j\}_{j=0}^n$ and direct computation it holds that
\begin{equation}
\label{eq:Sigma}
[\Phi W(\Phi W)^T]_{jl} = (P_j,P_l)_D = \begin{cases} \sum_{i=1}^N P_j(y_i)^2 w_i & \mbox{for $j=l$} \\ 0 & \mbox{otherwise}
\end{cases}. 
\end{equation}
That is, the matrix $\Phi W$ has a (thin) singular value decomposition $\Phi W = U \Sigma V^T$ for unitary $U\in \mathbb{R}^{ (n+1)\times (n+1)}$, $V\in \mathbb{R}^{N\times (n+1)}$, and diagonal  $\Sigma \in \mathbb{R}^{(n+1)\times(n+1)}$ s.t. $\Sigma_{jj} = \left( \sum_{i=1}^N P_j(y_i)^2w_i \right)^{1/2}$ (the singular values are not ordered). We proceed to estimate the smallest singular value $\sigma_{\rm min}$ from below. 

It holds by \eqref{eq:Sigma} and normalization of $\{ P_j\}$ that 
\begin{equation}
\sigma^2_{\rm min} = \min_{j\in \{0,\ldots,n\}} \sum_{i=1}^N P_j(y_i)^2w_i \geq \min_{q\in \mathcal{P}^n([-1,1])\setminus\{0\}} \frac{\sum_{i} q(y_i)^2 w_i}{\int_{-1}^1 q(y)^2\; dy}.
\end{equation}
Further, 
\begin{equation}
\label{eq:sigma_estimate}
\sigma^2_{\rm min} \geq \min_{q\in \mathcal{P}^n([-1,1])\setminus\{0\}} \frac{\int_{-1}^1 q(y)^2 \; dy + \sum_{i} q(y_i)^2 w_i - \int_{-1}^1 q(y)^2 \; dy }{\int_{-1}^1 q(y)^2\; dy} \geq 1 - e,
\end{equation} 
where $e$ is as defined in (\ref{eq:MZcond}).
Under the assumption $e<1$ it holds that $\sigma_{\rm min}>0$ and one solution to $\Phi \alpha(y) = \psi(y)$ is
\begin{equation*}
W^{-1}\alpha(y) = V \Sigma^{-1}U^T \psi(y).
\end{equation*}

Thus,  
\begin{equation}
\label{eq:Walpha_estimate1}
\| W^{-1} \alpha(y) \|^2_2 \leq \| \Sigma^{-1}\|^2_2 \| \psi(y)\|^2_2  =   \frac{\| \psi(y)\|^2_2 }{\sigma_{\rm min}^2}.
\end{equation}
By normalization of $\{P_j\}$ it holds that 
\begin{equation}
\label{eq:bynorm}
\int_{-1}^1 \| \psi(y) \|_2^2 \; dy = \int_{-1}^1 \sum_{j=0}^n P_j(y)^2 \; dy  = n+1.
\end{equation}
The proof is completed by combining \eqref{eq:Walpha_estimate1}, \eqref{eq:bynorm}, and \eqref{eq:sigma_estimate}. 
\end{proof}
Next Theorem uses Lemmas~\ref{lemma:cond} and \ref{lemma:existence} to show the existence of an approximation operator $\Pi_n$ to the snapshot subspace that is degree-$n$ polynomially exact  and establishes a stability estimate for it. 
\begin{theorem}
\label{thm:mainresult} Let $n\in\mathbb{N}$, $\{y_i\}_{i=1}^N \subset (-1,1)$, $\{w_i\}_{i=1}^N \subset \mathbb{R}^+$ satisfy Assumption~\ref{ass:ip} and $e \equiv e(\{y_i\}_{i=1}^N,\{w_i\}_{i=1}^N,n)$ be as defined in \eqref{eq:MZcond}. Assume that $e < 1$. In addition, let $\alpha \in \mathcal{P}^n([-1,1];\mathbb{R}^N)$ be as defined in Lemma~\ref{lemma:existence} and define the approximation operator $\Pi_n \in C([-1,1];\mathbb{R}^m_{A_0}) \rightarrow \mathcal{P}^n([-1,1];\mathbb{R}^m)$ as 
\begin{equation*}
(\Pi_n f)(y) =\sum_{i=1}^N \alpha_i(y) f(y_i).
\end{equation*}
Then $\Pi_n$ satisfies the following properties:
\begin{enumerate}
    \item $\Pi_n$ is degree-$n$ polynomially exact, i.e., $\Pi_n f = f$ for any $f\in \mathcal{P}^n([-1,1];\mathbb{R}^m)$.
    \item $\Pi_n$ satisfies the stability bound: $\|\Pi_n f \|_{L^2(-1,1;\mathbb{R}^m_{A_0})}
\leq C_{\rm s}(n) \|f\|_{L^\infty(-1,1;\mathbb{R}^m_{A_0})}$  for any $f\in C([-1,1];\mathbb{R}^m_{A_0})$ and 
\begin{equation}
\label{eq:stability_constant}
C_{\rm s}(n)= \left( \frac{2n+2}{1-e} \right)^{1/2}.
\end{equation}
\end{enumerate}
\end{theorem}
 
\begin{proof}
It holds that 
\begin{equation*}
\|\Pi_n f \|^2_{L^2(-1,1;\mathbb{R}^m_{A_0})}
= \int_{-1}^1 \left \| F \alpha(y)\right\|_{A_0}^2
= \int_{-1}^1 \left \| A_0^{1/2} F \alpha(y)\right\|_{2}^2,
\end{equation*}
for $F = \begin{bmatrix} f(y_1) & f(y_2) & \ldots &f(y_N) \end{bmatrix} \in \mathbb{R}^{m\times N}$. Then
\begin{equation}
\label{eq:stab_1}
\|\Pi_n f \|^2_{L^2(-1,1;\mathbb{R}^m_{A_0})}
\leq \| A_0^{1/2} F W \|^2_{2}
\int_{-1}^1 \| W^{-1} \alpha(y)\|_{2}^2.
\end{equation}
Recalling that $\| \cdot \|_2 \leq \| \cdot \|_F$ and using the definition of the $A_0$-norm gives
\begin{equation*}
\| A^{1/2}_0FW\|_2 \leq \|  A^{1/2}_0FW \|_F = \left(\sum_{i=1}^N \| f(y_i)\|_{A_0}^2 w_i \right)^{1/2}.
\end{equation*}
Estimating each $\| f(y_i)\|_{A_0}$ by the $L^\infty(-1,1;\mathbb{R}^m_{A_0})$-norm of $f$ and using the assumption $\sum_i w_i = 2$ gives
\begin{equation}
\label{eq:stab_2} 
\| A_0^{1/2}FW \|_{2} \leq \| f \|_{L^\infty(-1,1;A_0)} \left(\sum_{i=1}^Nw_i \right)^{1/2} = \sqrt{2} \| f \|_{L^\infty(-1,1;A_0)}. 
\end{equation}
Combining \eqref{eq:stab_1},\eqref{eq:stab_2}, and \eqref{eq:alpha_stab} completes the proof. 
\end{proof}

\begin{remark}
    The property (ii) is central, as it states that the stability constant (\ref{eq:stability_constant}) grows sub-linearly with rate $O(n^{1/2})$ with respect to the polynomial degree. 
\end{remark}
Using identical techniques as with the Lagrange interpolation operator gives the following near-best approximation result for $\Pi_n$. 
\begin{corollary} \label{cor:ip_error} Under the same assumptions as in Theorem~\ref{thm:mainresult} it holds that 
\begin{equation}
\label{eq:approximation_result}
\| f - \Pi_n f\|_{L^2(-1,1;\mathbb{R}^m_{A_0})}  \leq (\sqrt{2} + C_{\rm s}(n)) \min_{p_n \in \mathcal{P}^n([-1,1];\mathbb{R}^m)} \| f-p_n  \|_{L^\infty(-1,1;\mathbb{R}^m_{A_0})}.
\end{equation}
for any $f\in C([-1,1];\mathbb{R}^m_{A_0})$ and $C_{\rm s}(n)$ as in \eqref{eq:stability_constant}.
\end{corollary}
\begin{proof}
By Theorem~\ref{thm:mainresult} $\Pi_n$ is degree-$n$ polynomially exact and stable.  These two properties are enough to relate the approximation error $\| (I-\Pi) f \|_{L^2(-1,1;\mathbb{R}^m_{A_0})}$ to the best degree $n$ polynomial approximation of $f$ in $L^\infty(-1,1;\mathbb{R}^m_{A_0})$-norm. By degree-$n$ polynomial exactness it holds that 
\begin{equation*}
\| f - \Pi_n f\|_{L^2(-1,1;\mathbb{R}^m_{A_0})} = 
\| f - p_n + \Pi_n (f-p_n)\|_{L^2(-1,1;\mathbb{R}^m_{A_0})}
\end{equation*}
for any $p_n \in\mathcal{P}^n([-1,1];\mathbb{R}^m)$.
Using the triangle inequality, the standard estimate $\| f-p_n \|_{L^2(-1,1;\mathbb{R}^m_{A_0})} \leq \sqrt{2} \| f-p_n \|_{L^\infty(-1,1;\mathbb{R}^m_{A_0})}$, and Theorem~\ref{thm:mainresult} (ii) yields 
\begin{equation*}
\| f - \Pi_n f\|_{L^2(-1,1;\mathbb{R}^m_{A_0})}  \leq (\sqrt{2} + C_{\rm s}(n)) \min_{p_n \in \mathcal{P}^n([-1,1];\mathbb{R}^m)} \| f-p_n  \|_{L^\infty(-1,1;\mathbb{R}^m_{A_0})}.
\end{equation*}
\end{proof}
The main difference between \eqref{eq:approximation_result} and the Lagrange interpolation error estimate \eqref{eq:lag_best} is that the multiplicative constant appearing in~\eqref{eq:approximation_result} grows at sub-linear rate $O(\sqrt{n})$ with respect to polynomial degree under very mild assumptions on the point set $\{y_i\}_{i=1}^N$. Namely, $\{y_i\}_{i=1}^N$ must satisfy the condition~\eqref{eq:MZcond}. In comparison, the Lebesgue constant appearing in the Lagrange error estimate grows exponentially, e.g., for equidistant point sets. The improved multiplicative constant compared to Lagrange interpolation is possible because we use more function evaluations to reach the same accuracy, the approximation is not exact at points $\{y_i \}$, and the error estimate in \eqref{eq:approximation_result} is in the $L^2$-norm with respect to the parameter.
\begin{remark} The reason for giving the stability estimate and the approximation result \eqref{eq:approximation_result} in the $L^2(-1,1;\mathbb{R}^m_{A_0})$ - norm arises from equation~\eqref{eq:bynorm} where it is natural to estimate $\| \psi(y)\|_2^2$ in the $L^2$ norm. Alternatively, one could use the Nikolskii-type inequality: there exists $C_{\rm Nik} >0$ such that $\| p_n \|_{L^\infty(-1,1)} \leq C_{\rm Nik} n \| p_n \|_{L^2 (-1,1)}$ for any $p_n \in \mathcal{P}^n([-1,1])$, see \cite[Thm. 2.6. p.102]{DeLo:93}. This yields the estimate
\begin{equation*}
    \|W^{-1}\alpha \|^2_{L^\infty(-1,1;\mathbb{R}^N)} \leq C^2_{\rm Nik} \frac{ n^2 (n+1)}{1-e}. 
\end{equation*}
An appropriate modification of the proof of Theorem~\ref{thm:mainresult} gives the stability estimate: if $e<1$, we have that 
\begin{equation*}
    \| \Pi_n f\|_{L^\infty(-1,1;\mathbb{R}^m_{A_0})}\leq C_{\rm s,\infty}(n) \| f\|_{L^\infty(-1,1;\mathbb{R}^m_{A_0})},
\end{equation*}
for
\begin{equation*}
    C_{s,\infty}(n) = \sqrt{2}C_{\rm Nik} n \left( \frac{n+1}{1-e}\right)^{1/2}.
\end{equation*}

Using this stability estimate and the same techniques as in Corollary~\ref{cor:ip_error} leads to the $L^\infty$ error estimate 
\begin{equation*}
\| f - \Pi_n f\|_{L^\infty(-1,1;\mathbb{R}^m_{A_0})}  \leq (1 + C_{\rm s,\infty}(n)) \min_{p_n \in\mathcal{P}^n([-1,1];\mathbb{R}^m)} \| f-p_n  \|_{L^\infty(-1,1;\mathbb{R}^m_{A_0})}.
\end{equation*}
Note that $C_{\rm s,\infty}$ grows as $O(n^{3/2})$. 
\end{remark}

\subsection{Gaussian and Arbitrary Snapshot point sets}
Next, we study the Marcinkiewicz–Zygmund condition~\eqref{eq:MZcond} that guarantees the existence of an approximation operator for two families of snapshot point sets. Our aim is to characterize when \eqref{eq:MZcond} is valid for a given  snapshot point set and polynomial degree $n$. 

First, consider choosing snapshots as degree $N$ Gaussian quadrature points and $\{w_i\}$ as the corresponding quadrature weights. The Gaussian quadrature is exact for degree $2N-1$ polynomials, and thus $e = 0$ if the polynomial degree $n$ satisfies $2n\leq 2N-1$. We obtain the following result. 
\begin{lemma}
\label{lemma:stab_gauss} Let $\{y_i\}_{i=1}^N \subset (-1,1)$ and $\{w_j\}_{j=1}^N \in \mathbb{R}^+$ be the Gaussian quadrature points and weights on $(-1,1)$, respectively. Then 
\begin{equation*}
\max_{q\in \mathcal{P}^{n}([-1,1])\setminus\{0\}} \left| \frac{\int_{-1}^1 q(y)^2 \; dy-\sum_{i} q(y_i)^2 w_i}{\int_{-1}^1 q(y)^2 \; dy} \right| = 0,
\end{equation*}
for any $n \leq N-1$. 
\end{lemma}
Second, consider using arbitrary snapshot points. In this case, the weights are chosen such that the related quadrature corresponds to a Riemann sum approximation. The weights do not affect the approximation operator, but are needed to study the condition~\eqref{eq:cond0}. For notational simplicity, assume that
\begin{equation*}
-1 < y_1 < y_2 < \ldots < y_N < 1.
\end{equation*}
Decompose the interval $(-1,1)$ into subintervals $\{I_i\}_{i=1}^{N}$ as follows 
\begin{equation}
  \label{eq:wrs}
  \begin{aligned}
I_i &= \left[ \frac{y_{i-1}+y_{i}}{2}, \frac{y_{i}+y_{i+1}}{2} \right] \quad \mbox{for $i\in \{2,\ldots,N-1\}$},
\\
I_1 &= \left(-1,\frac{y_1+y_2}{2} \right], \quad \mbox{and} \quad I_{N} = \left[ \frac{y_N+y_{N-1}}{2},1 \right).
\end{aligned}
\end{equation} 
We choose $w_i = |I_i|$ so that $\sum_{i=1}^N f(y_i) w_i$ corresponds to the Riemann sum approximation of the integral $\int_{-1}^1 f$ for any $f\in C([-1,1])$. A suitable relative error estimate for this quadrature follows by making a small modification to the standard error analysis of the Riemann sum approximation and using the inverse inequality for polynomials: there exists $C_{\rm inv}>0$ such that 
\begin{equation}
\label{eq:inverse_inequality}
    \| q' \|_{L^2(-1,1)} \leq C_{\rm inv}n^2 \| q \|_{L^2(-1,1)},
\end{equation}
for any $q\in \mathcal{P}^n([-1,1])$ and $n \in \mathbb{N}$, see \cite{borwein1995polynomials}.
\begin{lemma}
\label{lemma:stab_random}
Let $-1 < y_1 < y_2 < \ldots < y_N <1$ and $w_i = |I_i|$ for $i=1,\ldots,N$ as defined in \eqref{eq:wrs}. In addition, let $h=\max_{i\in \{1,\ldots,N\}} |I_i|$. Then 
\begin{equation*}
\max_{q\in \mathcal{P}^{n}([-1,1])\setminus\{0\}} \left| \frac{\int_{-1}^1 q(y)^2 \; dy-\sum_{i=1}^N q(y_i)^2 w_i}{\int_{-1}^1 q(y)^2 \; dy} \right| \leq 2C_{\rm inv}hn^2,
\end{equation*}
where $C_{\rm inv}$ is the constant from the inverse inequality \eqref{eq:inverse_inequality}.
\end{lemma}
\begin{proof} As $w_i = |I_i|$, it holds that
\begin{equation}
    E=\left|\int_{-1}^1 q(y)^2 \; dy-\sum_{i=1}^N q(y_i)^2 w_i \right| = \left| \sum_{i=1}^{N}  \int_{I_i}  q(y)^2 -  q(y_i)^2 dy \right|.
\end{equation}
Set $F=q(y)^2$. By the fundamental theorem of calculus,
\begin{equation*}
    F(y) - F(y_i) = \int_{y_i}^y F'(s) ds,
\end{equation*}
Noting that $y \in I_i$ 
\begin{equation*} 
    \left| \int_{I_i} F(y) -  F(y_i)\, dy \right| =\left| \int_{I_i} \int_{y_i}^y F'(s) \, ds \, dy \right|  \leq |I_i| \int_{I_i} |F'(s)| \, ds. 
\end{equation*}
By the product rule and the Cauchy--Schwarz inequality,
\begin{equation*} 
 \int_{I_i} |F'(s)| \, ds 
\leq 2\lVert q' \rVert_{L^2(I_i)} \lVert q \rVert_{L^2(I_i)}. 
\end{equation*}
Combining yields the estimate
 \begin{equation*} 
    E \leq 2\sum_{i=1}^N |I_i| \lVert q' \rVert_{L^2(I_i)} \lVert q \rVert_{L^2(I_i)} \leq 2h \sum_{i=1}^N \lVert q' \rVert_{L^2(I_i)} \lVert q \rVert_{L^2(I_i)}.
\end{equation*}
Using the Cauchy--Schwarz inequality gives  
 \begin{equation*} 
E \leq 2h\left(\sum_{i=1}^N \lVert q' \rVert_{L^2((I_i))}^2 \right)^{1/2} \left(\sum_{i=1}^N \lVert q \rVert_{L^2(I_i)}^2 \right)^{1/2} = 2h\lVert q' \rVert_{L^2(-1,1)} \lVert q \rVert_{L^2(-1,1)} 
\end{equation*}
Using the inverse inequality in \eqref{eq:inverse_inequality} for polynomials completes the proof.
\end{proof}
\section{Error analysis}
\label{sec:error}

Next, we apply the approximation operator developed in Section~\ref{sec:approximation} to derive a priori error estimates for the Lagrange and POD ROMs of our model problem~\eqref{eq:plinsys}. Let $x(y)$ be the solution to \eqref{eq:plinsys} and $\tilde{x}(y)$ the subspace solution from the Lagrange subspace $V(\{y_i\}_{i=1}^N)$ as defined in  \eqref{eq:Lagrange}. Recall that the solution to the one-parametric model problem has the expansion
\begin{equation}
\label{eq:xfk}
    x(y) = \sum_{k=1}^m \gamma_k(y) \beta_k v_k \quad \mbox{for} \quad \gamma_k(y) = (1+\lambda_k y)^{-1}.
\end{equation} 
In this section, we estimate the error
\begin{equation}
\label{eq:L2error}
    \|x-\tilde{x}\|_{L^2(-1,1;\mathbb{R}^m_{A_0})} =\left( \int_{-1}^1 \| x(y)-\tilde{x}(y)\|^2_{A_0} \right)^{1/2}.
\end{equation}
The next theorem shows that if the sample points satisfy the condition $e < 1$ for $e\equiv e(\{y_i\}_{i=1}^N,\{w_i\}_{i=1}^N,n)$ as defined in \eqref{eq:MZcond}, the error \eqref{eq:L2error} is bounded by a multiplicative factor and the best degree $n$ polynomial approximation of $x$. 

\begin{theorem} \label{thm:mainresult2} Let $\{y_i\}_{i=1}^N \subset(-1,1)$ be the snapshot set and $\{w_i\}_{i=1}^N \subset \mathbb{R}^+$ positive weights satisfying Assumption~\ref{ass:ip}. In addition, let $x(y)$ be the solution to \eqref{eq:plinsys} and $\tilde{x}(y)$ the subspace solution from $V(\{y_i\}_{i=1}^N)$. Assume that $e < 1$ for $e\equiv e(\{y_i\}_{i=1}^N,\{w_i\}_{i=1}^N,n)$ as defined in \eqref{eq:MZcond}. Then there holds that 
\begin{equation*}
\|x-\tilde{x}\|_{L^2(-1,1;\mathbb{R}^m_{A_0})} \leq  \kappa \left( \sqrt{2} + C_{\rm s}(n) \right) \min_{q\in \mathcal{P}^n([-1,1];\mathbb{R}^m)} \|x -q\|_{L^\infty(-1,1;\mathbb{R}^m_{A_0})}.
\end{equation*}
where $\kappa$ and $C_{\rm s}(n)$ are as defined in \eqref{eq:kappadef} and \eqref{eq:stability_constant}, respectively.
\end{theorem}
\begin{proof}
The result follows from the near-best approximation result and Corollary~\ref{cor:ip_error}.
\end{proof}

Hence, the error of the Lagrange subspace solution depends on the accuracy of the quadrature rule related to the snapshot points and the best polynomial approximation error of the parameter-to-solution map $x$. Observe, that the quadrature rule has to be \emph{sufficiently accurate} for the snapshot solution to have an error that is controlled by the best approximation utilising degree $n$ polynomials. The relation between $n$ and the snapshot set is characterised for two families of snapshot point sets in  Lemmas~\ref{lemma:stab_gauss} and \ref{lemma:stab_random}. 

The best polynomial approximation of the parameter-to-solution map $x$ is established using $A_0$-orthogonality as in the proof of Lemma~\ref{cor:rembound} and the best approximation result for coefficient functions $\gamma_k$ in \eqref{eq:gamma_best}. To be precise,
\begin{equation*}
\min_{q\in \mathcal{P}^n([-1,1];\mathbb{R}^m)} \|x -q\|_{L^\infty(-1,1;\mathbb{R}^m_{A_0})} \leq \frac{4}{r^2} (1-r)^{n+1} \|x(0)\|_{A_0}.
\end{equation*}
Combining this bound with Theorem~\ref{thm:mainresult} and Lemmas~\ref{lemma:stab_gauss} and \ref{lemma:stab_random} immediately yields error estimates for Lagrange subspace ROM using Gaussian quadrature points or arbitrary snapshot points 
\begin{corollary}
\label{cor:error_gauss} Let $\{y_i\}_{i=1}^N$ and $\{w_j\}_{j=1}^N$ be the Gaussian quadrature points and weights on $(-1,1)$, respectively. Then 
\begin{equation*}
\frac{ \| x-\tilde{x}\|_{L^2(-1,1;\mathbb{R}^m_{A_0})}}{\|x(0)\|_{A_0}} \leq \kappa\left(\sqrt{2} + \sqrt{2(n+1)} \right) \frac{4}{r^2} (1-r)^{n+1} 
\end{equation*}
for any $n \leq N-1$. 
\end{corollary}
\begin{corollary}
\label{cor:error_random}
Let $\{y_i\}_{i=1}^N \subset (-1,1)$, $\{w_i\}_{i=1}^N \subset \mathbb{R}_+$ satisfy Assumption~\ref{ass:ip}. Let $h \equiv h(\{y_i\}_{i=1}^N)$ be as defined in Lemma~\ref{lemma:stab_random}. Assume that the point set and polynomial degree satisfy $2C_{\rm inv}hn^2 <1$. Then   
\begin{equation*}
\frac{\| x-\tilde{x} \|_{L^2(-1,1;\mathbb{R}^m_{A_0})}}{\|x(0)\|_{A_0}}  \leq  \kappa \left( \sqrt{2} +  \frac{\sqrt{2(n+1)}}{1-2C_{\rm inv}hn^2} \right) \frac{4}{r^2}  (1-r)^{n+1}. 
\end{equation*}
\end{corollary}
Corollary~\ref{cor:error_random} is our main a priori error estimate for the Lagrange subspace method. It applies to arbitrary point sets. The polynomial degree $n$ is determined by $h(\{y_i\}_{i=1}^N)$ that, intuitively speaking measures point density. For example, for equidistant point sets, $h$ behaves as $O(\frac{1}{N})$, and thus $n = O(\sqrt{N})$. This leads to sub-exponential convergence $O(N^{1/4} (1-r)^{\sqrt{N}})$ with respect to $N$. 

\subsection{Error Estimate for the POD Method}

Next, we apply the approximation operator given in Section~\ref{sec:approximation}  to study the error of the POD based ROM. We consider the POD subspace $V_{\rm POD}$ as in Definition \ref{def:VPOD}. Here it is important to use the same weights $\{w_i\}_{i=1}^N$ to define the POD subspace and to verify the condition $e < 1$, where $e$ is defined in (\ref{eq:MZcond}). Using the same technique as in the proof of Theorem~\ref{thm:poderror} but with the approximation operator $\Pi_n$, we obtain the following Theorem. 
\begin{theorem} \label{thm:poderror} Let $\{y_i\}_{i=1}^N \subset(-1,1)$ be the snapshot set and $\{w_i\}_{i=1}^N \subset \mathbb{R}^+$ be positive weights satisfying Assumption~\ref{ass:ip}. In addition, let $x(y)$ be the solution to \eqref{eq:plinsys} and $\tilde{x}_{\rm POD}(y)$ the subspace solution from $V_{\rm POD}$ defined using the weights $\{w_i\}_{i=1}^N$. Assume that $e<1$, where $e$ is defined in (\ref{eq:MZcond}).
Then there holds that 
\begin{equation}
\label{eq:PODW}
\kappa^{-1} \| x-\tilde{x}_{\rm POD}\|_{L^2(-1,1;\mathbb{R}^m_{A_0})} \leq \sqrt{2} \|(I-\Pi_n)x \|_{L^2(-1,1;\mathbb{R}^m_{A_0})}+ \sigma_{r+1} C_{\rm s}(n). 
\end{equation}
Here $\sigma_{r+1}$ and $C_{\rm s}(n)$ are as defined in \eqref{eq:def_sr} and \eqref{eq:stability_constant}.
\end{theorem}
\begin{proof}

By the near-best approximation property in \eqref{eq:A0approx}, the subspace solution $\tilde{x}_{\rm POD}(y)$ from $V_{\rm POD}$ satisfies 
\begin{equation*}
\kappa^{-2} \| x(y) - \tilde{x}_{\rm POD}(y) \|^2_{A_0} \leq \| x(y)-v_{\rm POD}\|^2_{A_0} \quad \mbox{for any $v_{\rm POD} \in V_{\rm POD}$}.
\end{equation*}
Choosing $v_{\rm POD} = \tilde{X}_r W^{-1}\alpha(y)$ for $\alpha(y)$ satisfying \eqref{eq:alpha_stab},  adding and subtracting $X\alpha(y)$, and applying triangle inequality gives
\begin{equation}
\label{eq:POD_intermediate}
\frac{\kappa^{-2}}{2}\| x(y) - \tilde{x}_{\rm POD}(y) \|^2_{A_0} \leq 
\| x(y) - X\alpha(y) \|^2_{A_0} + 
\| \tilde{X} - \tilde{X}_r \|^2_{2,A_0}   \|W^{-1}\alpha(y)\|^2_2. 
\end{equation}
Here we have used the identity $\tilde{X} = XW$. Observe, that $X\alpha(y) = \Pi_n x(y)$. Hence, 
\begin{equation*}
\| x(y) - X\alpha(y) \|_{A_0} = \| [(I- \Pi_n)x](y) \|_{A_0}.
\end{equation*}
The first term on the RHS of \eqref{eq:PODW} arises by integrating this term over $(-1,1)$.

We proceed to estimate the last term in \eqref{eq:POD_intermediate}. By Definition~\ref{def:VPOD}, it holds that $\| \tilde{X} - \tilde{X}_r\|_{2,A_0} = \sigma_{r+1}$ so that
\begin{equation*}
\int_{-1}^1 \| \tilde{X} - \tilde{X}_r \|^2_{2,A_0}   \|W^{-1}\alpha(y)\|^2_2 \; dy = \sigma_{r+1}^2 \int_{-1}^1 \| W^{-1} \alpha(y) \|_{2}^2 \; dy.  
\end{equation*}
Using the estimate $\int_{-1}^1 \| W^{-1} \alpha(y) \|_{2}^2 \; dy \leq \frac{C_{\rm s}(n)^2}{2}$ from Lemma \ref{lemma:existence}, integrating \eqref{eq:POD_intermediate} over $(-1,1)$, and multiplying with $2$ gives
\begin{equation*}
\kappa^{-2} \| x-\tilde{x}_{\rm POD}\|^2_{L^2(-1,1;\mathbb{R}^m_{A_0})} \leq 2 \int_{-1}^1\| [(I- \Pi_n)x](y) \|^2_{A_0} \; dy +  \sigma^2_{r+1} C^2_{\rm s}(n). 
\end{equation*}
Taking the square root completes the proof. 
\end{proof}

\section{Conclusions}

This article gives new a priori error estimates for Lagrange and POD ROMs that are based on arbitrary snapshot point sets. The key tool in deriving these estimates is a new approximation operator to the snapshot subspace given in Theorem~\ref{thm:mainresult}. This operator is defined to be exact for vector-valued degree-$n$ polynomial functions. Existence and stability of the operator are proven under the assumption that there exists a  quadrature rule with snapshot points and some positive weights that can integrate second powers of degree-$n$ polynomials with relative accuracy smaller than one, see~\eqref{eq:MZcond}. This condition turns out to be relatively mild as Lemma~\ref{lemma:stab_random} shows that it is satisfied for arbitrary snapshot points, if the point spacing satisfies $h=\mathcal{O}(n^{-2})$. In contrast to Lagrange interpolation whose error estimates feature the Lebesgue constant, which grows exponentially for equidistant point sets, the stability constant of our approximation operator grows at a sublinear rate $O(n^{1/2})$ with respect to the polynomial degree. Our POD and Lagrange ROM a priori error estimates in Theorem~\ref{thm:mainresult2} and Corollary~\ref{cor:error_random} are valid for any values of the stability and continuity constants $r$ and $C$ as defined in \eqref{eq:def_r} and \eqref{eq:a1_cont}, respectively. Parametric PDEs become much more complicated when the parameter dimension is high. These challenges are not studied in this work and they present an interesting direction for future work. 

\section{Acknowledgements} Vigdis Toresen's work was part of the Ministry of Education and Culture's Doctoral Education Pilot under Decision No. VN/3137/2024-OKM-6 (Doctoral Education Pilot for Mathematics of Sensing, Imaging and Modelling). This work was also supported by the Finnish Research Council (decisions 353080, 353081, 358944, 359181)

\bibliographystyle{plain}
\bibliography{references}

@book{DeLo:93,
author = {DeVore, Ronald A. and Lorentz, G. G.},
address = {Berlin},
title = {Constructive approximation},
language = {eng},
lccn = {93018420},
publisher = {Springer-Verlag},
year = {1993},
}

@book{trefethen2019,
  title={Approximation Theory and Approximation Practice, Extended Edition},
  author={Trefethen, Lloyd N.},
  year={2019},
  publisher={SIAM},
  address={Philadelphia, PA}
}

@article{Cohen2017,
  author  = {Cohen, Albert and Migliorati, Giovanni},
  title   = {Optimal weighted least-squares methods},
  journal = {The SMAI journal of computational mathematics},
  year    = {2017},
  volume  = {3},
  pages   = {181--203},
  doi     = {10.5802/smai-jcm.24}
}

@article{Cohene2013,
  author  = {Cohen, Albert and Davenport, Mark A. and Leviatan, Dany},
  title   = {On the Stability and Accuracy of Least Squares Approximations},
  journal = {Foundations of Computational Mathematics},
  year    = {2013},
  volume  = {13},
  pages   = {819--834},
  doi     = {10.1007/s10208-013-9142-3}
}

@article{An2022,
  author  = {An, Congpei and Wu, Hao-Ning},
  title   = {On the quadrature exactness in hyperinterpolation},
  journal = {BIT Numerical Mathematics},
  year    = {2022},
  volume  = {62},
  pages   = {1899--1919},
  doi     = {10.1007/s10543-022-00935-x}
}

@book{davis,
  title={Interpolation and Approximation},
  author={Davis, P.J.},
  isbn={9780486624952},
  lccn={75002568},
  year={1963},
  publisher={Dover Publications}
}

@article{An2025,
  author  = {An, Congpei and Ran, Jiashu and Wu, Hao-Ning},
  title   = {The path of hyperinterpolation: A survey},
  journal = {Dolomites Research Notes on Approximation},
  year    = {2025},
  volume  = {18},
  pages   = {135--145},
}

@book{gautschi2004orthogonal,
  title={Orthogonal Polynomials: Computation and Approximation},
  author={Gautschi, Walter},
  series={Numerical Mathematics and Scientific Computation},
  year={2004},
  publisher={Oxford University Press},
  isbn={9780198506720}
}

@article{brutman1997lebesgue,
  author  = {Brutman, Lev},
  title   = {Lebesgue functions for polynomial interpolation---a survey},
  journal = {Annals of Numerical Mathematics},
  volume  = {4},
  number  = {1--4},
  pages   = {111--127},
  year    = {1997}
}

@article{FiRh:83,
  author  = {Fink, J. P. and Rheinboldt, W. C.},
  title   = {On the error behavior of the reduced basis technique for nonlinear finite element approximations},
  journal = {ZAMM - Journal of Applied Mathematics and Mechanics / Zeitschrift f{\"u}r Angewandte Mathematik und Mechanik},
  year    = {1983},
  volume  = {63},
  pages   = {21--28}
}

@article{Po:85,
  title={Estimation of the Error in the Reduced Basis Method Solution of Nonlinear Equations},
  author={Porsching, T. A.},
  journal={Mathematics of Computation},
  volume={45},
  number={172},
  pages={487--496},
  year={1985},
  publisher={American Mathematical Society},
  url={https://www.jstor.org/stable/2008138}
}

@article{ScGi:11,
  title={Sparse tensor discretizations of high-dimensional parametric and stochastic PDEs},
  author={Schwab, Christoph and Gittelson, Claude Jeffrey},
  journal={Acta Numerica},
  volume={20},
  pages={291--467},
  year={2011},
  publisher={Cambridge University Press}
}

@phdthesis{Ve:03,
  title={Reduced Basis Methods Applied to Problems in Elasticity: Analysis and Applications},
  author={Veroy, Karen},
  year={2003},
  school={Massachusetts Institute of Technology},
  address={Cambridge, MA}
}

@book{BrSc:08,
  title={The Mathematical Theory of Finite Element Methods},
  author={Brenner, Susanne C and Scott, L Ridgway},
  edition={3rd},
  year={2008},
  publisher={Springer},
  series={Texts in Applied Mathematics},
  volume={15},
  doi={10.1007/978-0-387-75934-0}
}

@article{LaRo:10,
title = {Parametric free-form shape design with PDE models and reduced basis method},
journal = {Computer Methods in Applied Mechanics and Engineering},
volume = {199},
number = {23},
pages = {1583-1592},
year = {2010},
issn = {0045-7825},
doi = {https://doi.org/10.1016/j.cma.2010.01.007},
url = {https://www.sciencedirect.com/science/article/pii/S0045782510000162},
author = {Toni Lassila and Gianluigi Rozza},
}

@article{MaQuRo:12,
  title={Shape optimization for viscous flows by reduced basis methods and free-form deformation},
  author={Manzoni, Andrea and Quarteroni, Alfio and Rozza, Gianluigi},
  journal={International Journal for Numerical Methods in Fluids},
  volume={70},
  number={5},
  pages={646--670},
  year={2012},
  publisher={John Wiley \& Sons, Ltd},
  doi={10.1002/fld.2712}
}

@article{LiWiGh:10,
  title={Parameter and state model reduction for large-scale statistical inverse problems},
  author={Lieberman, Chad and Willcox, Karen and Ghattas, Omar},
  journal={SIAM Journal on Scientific Computing},
  volume={32},
  number={5},
  pages={2523--2542},
  year={2010},
  publisher={SIAM},
  doi={10.1137/090775622}
}

@incollection{FaMaYoWiBl:11,
  title={Surrogate and reduced-order modeling: a comparison of approaches for large-scale statistical inverse problems},
  author={Frangos, Michalis and Marzouk, Youssef and Willcox, Karen and van Bloemen Waanders, Bart},
  booktitle={Large-Scale Inverse Problems and Quantification of Uncertainty},
  editor={Biegler, Lorenz and Biros, George and Ghattas, Omar and Heinkenschloss, Matthias and Keyes, David and Mallick, Bani and Marzouk, Youssef and Tenorio, Luis and van Bloemen Waanders, Bart and Willcox, Karen},
  pages={123--149},
  year={2011},
  publisher={John Wiley \& Sons},
  address={Hoboken, NJ, USA},
  doi={10.1002/9780470685853.ch7}
}

@article{BaNoTe:07,
  title={A stochastic collocation method for elliptic partial differential equations with random input data},
  author={Babu{\v{s}}ka, Ivo and Nobile, Fabio and Tempone, Ra{\'u}l},
  journal={SIAM Journal on Numerical Analysis},
  volume={45},
  number={3},
  pages={1005--1034},
  year={2007},
  publisher={SIAM},
  doi={10.1137/050645142}
}

@article{MaPaTu:02,
     author = {Yvon Maday and Anthony T. Patera and G. Turinici},
     title = {Global a priori convergence theory for reduced-basis approximations of single-parameter symmetric coercive elliptic partial differential equations},
     journal = {Comptes Rendus. Math\'ematique},
     pages = {289--294},
     year = {2002},
     publisher = {Elsevier},
     volume = {335},
     number = {3},
     doi = {10.1016/S1631-073X(02)02466-4},
     language = {en}
}

@article{KaVo:07,
  title={Galerkin proper orthogonal decomposition methods for parameter dependent elliptic systems},
  author={Kahlbacher, Martin and Volkwein, Stefan},
  journal={Discussiones Mathematicae, Differential Inclusions, Control and Optimization},
  volume={27},
  number={1},
  pages={95--117},
  year={2007},
  url={http://eudml.org/doc/271156}
}

@article{CoDe:15,
  title={Approximation of high-dimensional parametric PDEs},
  author={Cohen, Albert and DeVore, Ronald},
  journal={Acta Numerica},
  volume={24},
  pages={1--159},
  year={2015},
  publisher={Cambridge University Press},
  doi={10.1017/S0962492915000033}
}

@article{BestUniformApproxZY,
title = {Best uniform approximation to a class of rational functions},
journal = {Journal of Mathematical Analysis and Applications},
volume = {334},
number = {2},
pages = {909-921},
year = {2007},
issn = {0022-247X},
doi = {10.1016/j.jmaa.2006.10.047},
author = {Zhitong Zheng and Jun-Hai Yong},
}

@article{BestUniformApproxJM,
title = {The best approximation of some rational functions in uniform norm},
journal = {Applied Numerical Mathematics},
volume = {55},
number = {2},
pages = {204-214},
year = {2005},
issn = {0168-9274},
doi = {https://doi.org/10.1016/j.apnum.2005.02.005},
author = {Sadegh Jokar and Bahman Mehri},
}

@book{Rivlin,
    author = {Rivlin, Theodore J.},
    title = {An Introduction to the Approximation of Functions},
    publisher = {Dover Publications},
    year = {2003}
}

@article{BiCoDa:11,
  title={Convergence Rates for Greedy Algorithms in Reduced Basis Methods},
  author={Binev, Peter and Cohen, Albert and Dahmen, Wolfgang and DeVore, Ronald and Petrova, Guergana and Wojtaszczyk, Przemyslaw},
  journal={SIAM Journal on Mathematical Analysis},
  volume={43},
  number={3},
  pages={1457--1472},
  year={2011},
  publisher={SIAM}
}

@article{BuMaPa:12,
      title={A priori convergence of the greedy algorithm for the parametrized reduced basis method},
      author={Buffa, Andrea and Maday, Yvon and Patera, Anthony T and Prud'homme, Christophe and Turinici, Gabriel},
      journal={ESAIM: Mathematical Modelling and Numerical Analysis},
      volume={46},
      number={3},
      pages={595--603},
      year={2012},
      publisher={EDP Sciences}
    }

@book{borwein1995polynomials,
  title     = {Polynomials and Polynomial Inequalities},
  author    = {Borwein, Peter and Erdélyi, Tamás},
  series    = {Graduate Texts in Mathematics},
  volume    = {161},
  year      = {1995},
  publisher = {Springer},
  address   = {New York, NY},
  isbn      = {978-1-4612-0793-1},
  doi       = {10.1007/978-1-4612-0793-1}
}

\appendix

\section{Best uniform approximation by polynomials}
\label{sec:best_poly}

We here review a classic result from approximation theory based on a Chebyshev series expansion. The Chebyshev polynomial of degree $n$ on $[-1,1]$ is defined by $T_n(t) = \cos (n\theta)$, where $t=\cos (\theta)$  and $\theta \in [0,\pi ]$.

%%%%%%%%%%%%%%%%%%%%%%%%%%%%%%%%%%%
\begin{lemma}\label{thm:BestUniform} \cite{Rivlin}\cite{BestUniformApproxJM} %\cite{BestUniformApproxZY}
    The best approximation by a polynomial of degree $\leq n$ in the uniform norm to $f(t)=1/(t - c)$ on $[-1,1]$, where $c>1$ is
    \begin{equation}
        p^\ast_n(t) = \frac{-2\tau}{(\tau^2-1) } +  \frac{4\tau}{(\tau^2-1) }\sum_{j=0}^{n-1}\tau^j T_j (t) - \frac{4\tau^{n+1}}{(\tau^2-1)^2}  T_n(t),
    \end{equation}
and
\begin{equation}\label{minimax_error}
   \left\|f - p^\ast_n  \right\|_{L^\infty(-1,1)} = \frac{4|\tau|^{n+2}}{(\tau^2 -1)^2}
\end{equation}
where $\tau = c-\sqrt{c^2-1}$. %$|\tau | < 1$.
\end{lemma}

\begin{remark}
    By some algebraic manipulation it can be shown that the error can equivalently be written as
    \begin{equation}\label{minimax_error}
   \left\|f - p^\ast_n  \right\|_{L^\infty(-1,1)} = \frac{(c-\sqrt{c^2-1})^{n}}{(c^2 -1)^2}.
\end{equation}
\end{remark}

%%%%%%%%%%%%%%%%%%%%%%%%%%%%%%%%%%%%%%%
\end{document}